\documentclass[12pt]{amsart}
\usepackage{amssymb,amsmath,mathrsfs}
\usepackage{hyperref}
\usepackage{epsfig,graphicx,color,mathrsfs}
\usepackage[english]{babel}
\usepackage[left=2.7cm,right=2.7cm,top=3cm,bottom=3cm]{geometry}

\newcommand{\R}{{\mathbb R}}

\newcommand{\tu}{\tilde{u}}

\numberwithin{equation}{section}
\newtheorem{theorem}{Theorem}[section]
\newtheorem{proposition}[theorem]{Proposition}
\newtheorem{lemma}[theorem]{Lemma}
\newtheorem{corollary}[theorem]{Corollary}
\newtheorem{definition}[theorem]{Definition}
\newtheorem{remark}[theorem]{Remark}
\newtheorem*{basic}{Condition (PE)}
\theoremstyle{definition}

\newcommand{\brm}{\begin{remark}\rm}
\newcommand{\erm}{\end{remark}}
\newcommand{\brms}{\begin{remark}\rm}
\newcommand{\erms}{\end{remark}}
\newcommand{\bte}{\begin{theorem}}
\newcommand{\ete}{\end{theorem}}
\newcommand{\bpr}{\begin{proposition}}
\newcommand{\epr}{\end{proposition}}
\newcommand{\ble}{\begin{lemma}}
\newcommand{\ele}{\end{lemma}}
\newcommand{\beq}{\begin{equation}}
\newcommand{\eeq}{\end{equation}}
\newcommand{\bdm}{\begin{displaymath}}
\newcommand{\edm}{\end{displaymath}}
\numberwithin{equation}{section}

\newcommand{\bos}{\begin{remark}\rm}
\newcommand{\eos}{\end{remark}}

\newcommand{\ben}{\begin{enumerate}}
\newcommand{\een}{\end{enumerate}}

\newcommand{\n }{\nabla }

\newcommand{\be}{\begin{equation}}
\newcommand{\ee}{\end{equation}}

\title[Monotonicity  in half-spaces]{Monotonicity of solutions of
 quasilinear degenerate elliptic equation in half-spaces}

\author[A.\ Farina]{Alberto Farina$^+$}
\address{Universit\'e de Picardie Jules Verne
\newline\indent
LAMFA, CNRS UMR 6140\newline\indent
Amiens, France}
\email{alberto.farina@u-picardie.fr}

\author[L.\ Montoro]{Luigi Montoro$^*$}
\address{Dipartimento di Matematica
\newline\indent
Universit\`a della Calabria
\newline\indent
Ponte Pietro Bucci 31B, I-87036 Arcavacata di Rende, Cosenza, Italy}
\email{montoro@mat.unical.it}

\author[B.\ Sciunzi]{Berardino Sciunzi$^*$}
\address{Dipartimento di Matematica
\newline\indent
Universit\`a della Calabria
\newline\indent
Ponte Pietro Bucci 31B, I-87036 Arcavacata di Rende, Cosenza, Italy}
\email{sciunzi@mat.unical.it}

\thanks{\it 2000 Mathematics Subject
 Classification: 35B05,35B65,35J70}

\thanks{$^+$Universit\'e de Picardie Jules Verne,
LAMFA, CNRS UMR 7352,
Amiens, France,
E-mail:{\em alberto.farina@u-picardie.fr}}

\thanks{$^*$Dipartimento di Matematica,
Universit\`a della Calabria,
Ponte Pietro Bucci 31B, I-87036 Arcavacata di Rende, Cosenza, Italy,
E-mail: {\em montoro@mat.unical.it}, {\em sciunzi@mat.unical.it}}

\thanks{AF and BS  were partially supported by ERC-2011-grant: \emph{Elliptic PDE's and symmetry of interfaces and layers for odd nonlinearities.}}

\thanks{LM and BS  were partially supported by PRIN-2011: {\em Variational and Topological Methods in the Study of Nonlinear Phenomena}}

\begin{document}

\begin{abstract}
We prove a weak comparison principle in narrow unbounded domains   for   solutions to $-\Delta_p u=f(u)$ in the case $2<p< 3$ and $f(\cdot)$ is a power-type nonlinearity, or in the case $p>2$ and $f(\cdot)$ is super-linear.
 We exploit it to prove  the monotonicity of positive  solutions to $-\Delta_p u=f(u)$ in half spaces (with zero Dirichlet assumption) and therefore to prove some Liouville-type theorems.
\end{abstract}

\maketitle

\tableofcontents

\medskip
\section{Introduction and statement of the  main results.}\label{introdue}

In this paper we consider the problem
\begin{equation}\label{E:P}
\begin{cases}
-\Delta_p u=-{\hbox {\rm div}} (|\nabla u|^{p-2} \nabla u )=f(u), & \text{ in }\mathbb{R}^N_+\\
u(x',y) \geqslant 0, & \text{ in } \mathbb{R}^N_+\\
u(x',0)=0,&  \text{ on }\partial\mathbb{R}^N_+
\end{cases}
\end{equation}
where $N \geq 2$ and we denote a generic point belonging to $\mathbb{R}^N_+$ by $(x',y)$ with $x'=(x_1,x_2, \ldots, x_{N-1})$ and $y=x_N$. It is well known that solutions of $p$-Laplace equations are generally of class $C^{1,\alpha}$ (see \cite{Di,Li,T}), and the equation has to be understood in the weak sense. \\

Our aim is to study monotonicity properties of the solutions. The crucial point to achieve  such a result is to obtain
weak comparison principles in narrow domains. These in fact allow to exploit  the Alexandrov-Serrin moving plane method \cite{A,S,GNN,BN}.

We refer the readers to \cite{BCN1,BCN2,BCN3,Dan1,Dan2,Fa,DaGl,FV2,QS}  for previous results concerning  monotonicity of the solutions
in half-spaces, in the non-degenerate case.
In particular we refer to the founding  papers of H.~Berestycki, L.~Caffarelli and L.~Nirenberg that influenced and inspired the subsequent literature, and also to the papers of E. Dancer.\\

\noindent When considering the $p$-Laplace operator, some results in this direction  have been obtained by the authors in \cite{FMS} for the case $\frac{2N+2}{N+2}<p<2$ (the case $p=2$ being well studied).\\

\noindent This paper is devoted to the case $p>2$ that turns out to be very much more complicated.
We will prove our results assuming one of the following:

\begin{itemize}
\item[($f_1$)]
 $f\in C^1(\mathbb{R}^+\cup \{0\})\cap C^2(\mathbb{R}^+)$ and, given $\mathcal{M}>0$, there exist $a=a(\mathcal{M})>0$ and $A=A(\mathcal{M})>0$ such that
\[
a\, s^q\leq f(s)\leq A\, s^q\qquad \text{and } \qquad  \big |f'(s)\big |\leq A \,s^{q-1}\qquad\qquad \text{ in} \,\, [0\,,\, \mathcal{M}]
\]
for some $q> p-1$.  In the case $p-1<q<2$ we further assume that there exists a constant $\tilde A>0$ such that, for any $0<t<s$, it follows:
 \[
 \frac{f(s)-f(t)}{s^q-t^q}\leq \tilde A\,.
 \]

\item[($f_2$)] The nonlinearity $f$ is \emph{positive} ($f(s)>0$ for $s>0$) and Locally Lipschitz continuous, with
 \[
 f(s)\geq c_f s^{p-1}\qquad \text{in}\quad [0\,,\,s_0]
 \]
 for some $s_0>0$ and some positive constant $c_f$.
\end{itemize}

\noindent Nonlinearities that satisfy  $(f_1)$ are referred to as power-type nonlinearities. As examples of nonlinearities that satisfy ($f_2$) one can consider exponential nonlinearities or  nonlinearities like $f(s)=(1+s)^q$, or super-linear power nonlinearities.\\

\noindent It is convenient to resume our assumptions as follows:\\

\begin{itemize}
\item[($H_1$)] The nonlinearity $f$ satisfies ($f_1$), and  $2<p<3$.
\item[($H_2$)] The nonlinearity $f$ satisfies ($f_2$), and  $p> 2$\,.
\end{itemize}
\begin{theorem}\label{th:wcpstrip}
Let   $u \in C^{1, \alpha}_{loc}(\Sigma_{(\lambda, \beta)})$ and $ v \in C^{1, \alpha}_{loc}(\Sigma_{(\lambda-2\bar\delta, \beta+2\bar\delta)})$ satisfy  $u,\nabla u \in L^{\infty}(\Sigma_{(\lambda, \beta)})$ and $ v,\nabla v \in L^{\infty}(\Sigma_{(\lambda-2\bar\delta, \beta+2\bar\delta)})$, where $\bar \delta>0$, $\Sigma_{(\lambda, \beta)}:= \left\{ \mathbb{R}^{N-1}\times [\lambda, \beta]\right \}$ and $0\leq\lambda<\beta$.
Assume that ($H_1$) holds and let $u$ be non-negative and $v$ be positive such that:
\begin{equation}\label{Eq:WCP}
\begin{cases}
-\Delta_p u= f(u) & \text{ in }\Sigma_{(\lambda, \beta)},\\
- \Delta_p v= f(v) & \text{ in }\Sigma_{(\lambda-2\bar\delta, \beta+2\bar\delta)},\\
\qquad u\leq v & \text{ on } \partial \Sigma_{(\lambda, \beta)}.
\end{cases}
\end{equation}
Assume furthermore that
 there exists a constant $C=C(p,u,v,f,N)>0$ such that, for any $x_0'\in\mathbb{R}^{N-1}$, it follows
\begin{equation}\label{condcrucial}
\begin{split}
&\int_{\mathcal{K}(x'_0)} \frac{1}{|\nabla v|^{\tau}}\frac{1}{|x-y|^\gamma}\leq C\,\beta^{(N+\tau-2p-\gamma)}\,v_0^{2p-2q-2-\tau}
\end{split}
\end{equation}
where $v_0= v(x'_0,\frac{\beta+\lambda}{2})$ and
 $\mathcal{K}(x'_0)$ is defined by\footnote{Note that here $B_r(x')$ is the ball in $\mathbb{R}^{N-1} $of radius $r$ centered at $x'$.} $ \mathcal {K}(x'_0)= B_{(\beta-\lambda)\sqrt{N}}(x'_0) \times (\lambda, \beta)$,  $\gamma< N-2$ if $N\geq 3$, or $\gamma=0$ if $N=2$ and $\max\{(p-2)\, ,\,0\}\leqslant \tau<p-1$.

Then there exists $d_0=d_0(p,u,v, f,N)>0$ such that\footnote{ $d_0$ will actually depend on the Lipschitz constant $L_f$  of $f$ in the interval $[-\max\{\|u\|_\infty,\|v\|_\infty\} ,\max\{\|u\|_\infty,\|v\|_\infty\}]$. }  if,  $0~<~\beta-\lambda<~ d_0$, it follows that
\begin{equation}\label{gfgfdtrscvbnmzzxs}
u \leq v \qquad \text{ in } \Sigma_{(\lambda, \beta)}\,.
\end{equation}
On the other hand, if we assume that $f$ is any positive ($f(s)>0$ for $s>0$) locally Lipschitz nonlinearity and
$\lambda>\underline{\lambda}>0$ and
 $v\geq \underline{v}>0$ in $\Sigma_{(\lambda-2\bar\delta, \beta+2\bar\delta)}$, then \eqref{gfgfdtrscvbnmzzxs} follows for any $p>2$.

\end{theorem}

As already pointed out, Theorem \ref{th:wcpstrip} is motivated by the application to the study of the monotonicity of the solutions to problem \eqref{E:P} and it will be exploited as stated in Corollary \ref{th:wcpstripbg}. \\

\noindent In the semilinear non-degenerate case, weak comparison principles are equivalent to weak maximum principles, that  in narrow (possibly unbounded) domains can be proved  arguing as in \cite{BCN1}. \\

\noindent Considering the $p$-Laplace operator,  we have to deal with two obstructions: the \emph{degenerate nature} of the  $p$-Laplace operator and the fact that it is a \emph{nonlinear operator}. \\
The degeneracy of the operator causes many technical difficulties and  the fact that solutions are not $C^2$ solutions.
In particular, in the case $p>2$ that we are considering, the standard Sobolev embedding has to be substituted by a weighted version in weighted Sobolev spaces.\\
\noindent Also, the fact that the $p$-Laplace operator is nonlinear causes that weak comparison principles are not equivalent to weak maximum principles.\\

\noindent A result similar to Theorem \ref{th:wcpstrip}, in the case when $1<p<2$,  was proved by the authors in \cite{FMS} (see Theorem 1.1), and was actually the first weak comparison principle in narrow unbounded domains for $p$-Laplace equations.\\

 Here we continue  the study started in \cite{FMS} considering the more difficult case $p>2$.
 To do this (and to prove Theorem \ref{th:wcpstrip}) we will go beyond the  technique introduced in \cite{FMS} (see  Theorem 1.1 in \cite{FMS}), taking care of the degeneracy
 of the weight $|\nabla u|^{p-2}$ that vanishes on the critical points of the solution since $p>2$.

 The proofs are based on the use
of the Poincar\'{e} inequality and an iteration scheme which makes use  of a particular choice of test-functions.\\
 \noindent The case $p>2$ is more complicated than the case
$1<p<2$ since the use of the classic Poincar\'{e} inequality has to be replaced by the use of a \emph{weighted} Poicar\'{e} type inequality, in the spirit of \cite{DS1}.
However, the constants in the \emph{weighted} Poicar\'{e} type inequality developed in \cite{DS1} depend upon the minimum of the solution $u$ (via $f(u)$) in the considered domain. Consequently the  Poicar\'{e}  constant may blow-up if $u$ approaches zero that may occur
since we do not make any a-priori assumptions on $u$.\\

\noindent Our effort here (in the proof of Theorem \ref{th:wcpstrip}) is to deal with this phenomenon. To do this, in Section \ref{section3} and Section \ref{section4}, we provide
a quantitative version of the \emph{weighted} Poicar\'{e} type inequality developed in \cite{DS1}. In some sense, we measure
how the Poincar\'{e} constant blow-up, when $u$ approaches zero. \\

\noindent This will be used taking also into account the fact that, in case when ($f_1$) holds, when $u$ approaches zero, also $f(u)$ (and $f'(u)$) approaches zero as well,
with a decay depending on the power-like nature of $f(\cdot)$ and this is of some advantage as it will be clear in the proofs.
The competition of this two phenomena, gives rise to the condition $p<3$. \\

\noindent We actually do not know if $p<3$ is or not a sharp condition (in the case of power-type nonlinearities). We only remark that $p<3$ is the sharp condition in order to get the $W^{2,2}_{loc}$ regularity of the solutions.

\noindent Let us now provide some applications that follow once that  Theorem \ref{th:wcpstrip}  is available that is

\begin{theorem}\label{mainthm}
Let $u \in C^{1,\alpha}_{loc}({\overline {{\mathbb{R}^N_+}}})$  be a positive solution of \eqref{E:P} with  $|\n u| \in L^{\infty}(\mathbb{R}^N_+)$ and assume that ($H_1$) holds. \\

 \noindent Then $u$ is monotone increasing w.r.t. the $x_N$-direction with
 \[
\frac{\partial u}{\partial x_N}\,>\,0 \quad \text{in}\quad \mathbb{R}^N_+,
 \]
 and consequently $u \in C^{2}({\overline {{\mathbb{R}^N_+}}})$.\\
 \noindent If moreover  $N=3$  and  $u\in L^\infty(\mathbb{R}^3_+)$,  then $u$ has one-dimensional symmetry\footnote{The case $N=2$ has been already considered in \cite{DS3}.} with $u(x',x_N)=u(x_N)$.
\end{theorem}

The proof of Theorem \ref{mainthm} is based on a refined version of   Alexandrov-Serrin moving plane method \cite{A,S,GNN,BN} that takes into account the lack of compacteness, caused by the fact that we work in unbounded domains.
 We refer to \cite{DP,DS1} for the adaptation of  the moving plane technique to the case of the $p$-Laplace operator in \emph{bounded} domains. Considering the case when the domain is  the half-space, the application of the moving plane technique is much more delicate since weak comparison principles in small domains have to be substituted by weak comparison principles in narrow unbounded domains. This causes  that there are no general results in the literature when dealing with the case of the $p$-Laplace.
 In \cite{DS3} it is considered the \emph{two dimensional} case for  positive solutions  of $-\Delta_p u=f(u)$ with a positive  nonlinearity  $f$.

The strength of Theorem  \ref{mainthm} is that it is proved without a-priori assumptions on the behavior of the solution. Namely  at infinity the solution may decay at zero in some regions, while it can be far from zero in some other regions. It is implicit (we will give some details of the proof) in any case in the proof of Theorem  \ref{mainthm} the following:

\begin{theorem}\label{mainthmfdfdfdfdf}
Let $u \in C^{1,\alpha}_{loc}({\overline {{\mathbb{R}^N_+}}})$  be a positive solution of \eqref{E:P} with  $|\n u| \in L^{\infty}(\mathbb{R}^N_+)$.  Assume that $$p>2$$ and $f$ is positive ($f(s)>0$ for $s>0$) and locally Lipschitz continuous. Assume furthermore that
\begin{equation}\label{gfgfgsdhjsjbvbbvbhcnncndnndhuhu}
u\geq \underline{u}_\beta>0\qquad \text{in}\quad \{y\geq\beta\}\,,
\end{equation}
for some $\beta>0$ and some positive constant $\underline{u}_\beta\in \mathbb{R}^+$.\\

 \noindent Then $u$ is monotone increasing w.r.t. the $x_N$-direction with
 \[
\frac{\partial u}{\partial x_N}\,>\,0 \quad \text{in}\quad \mathbb{R}^N_+,
 \]
 and consequently $u \in C^{2}({\overline {{\mathbb{R}^N_+}}})$.\\

 \noindent The result in particular follows (without assuming \eqref{gfgfgsdhjsjbvbbvbhcnncndnndhuhu}) if ($H_2$) holds,
 since \eqref{gfgfgsdhjsjbvbbvbhcnncndnndhuhu} is satisfied in this case (by Lemma \ref{le:cuccurucucu}).
\end{theorem}

The monotonicity of the solution is an important information, which in particular implies the stability of the solution, see  \cite{DFSV,FSV}.
Note in particular that in  many cases the  monotonicity of the solution (and the stability of the solution) can be exploited to deduce Liouville type theorems.

Following \cite{Fa2,DFSV}, we  set

\begin{equation}\label{**B}
q_c(N,p)=\frac{[(p-1)N-p]^2+p^2(p-2) -p^2(p-1)N+2p^2\sqrt{(p-1)(N-1)}}{(N-p)[(p-1)N-p(p+3)]}\,,
\end{equation}

\noindent that is the critical exponent, that was found in \cite{Fa2} in the case $p=2$, and later introduced in \cite{DFSV} for the case $p>2$.
This exponent is critical in the sense that it gives the sharp condition for the existence (or non-existence) of stable solution of
Lane-Emden-Fowler type equations. We refer to \cite{DFSV} for the definition of stable solutions in our setting, and for a proof of the fact that
monotone solutions are actually stable solutions.\\
The exponent $q_c(N,p)$ is larger than the classic  critical exponent arising from Sobolev embedding.   We refer the reader to  \cite{Zou}
previous Liouville type results for $p$-Laplace equations.\\

\noindent We have the following

\begin{theorem}\label{liouvillenextgenerationtris}
Let  $2<p<3$ and consider $u\in C^{1}(\R^N_+)$ a non-negative
weak solution of~\eqref{E:P} in~$\R^N_+$ with $|\nabla u|\in L^\infty (\R^N_+)$ and
\[
f(s)=s^q\,.
\]
Assume that
$$
\begin{cases}
&(p-1) < q <\infty,\quad\quad\quad \,\,\,\,\text{if }\qquad \displaystyle N \leqslant  \frac{p(p+3)}{p-1},\\
&(p-1) < q < q_c(N,m), \quad \text{if }\qquad \displaystyle  N >\frac{p(p+3)}{p-1}\,,\end{cases}$$

\noindent then $u=0$.\\

\noindent If moreover we assume that $u$ is bounded, then it follows that $u=0$ assuming only that

$$
\begin{cases}
&(p-1) < q <\infty,\quad\quad\quad \quad\qquad\,\,\,\,\text{if }\qquad \displaystyle(N-1) \leqslant  \frac{p(p+3)}{p-1},\\
&(p-1) < q < q_c((N-1),m), \quad \,\,\text{if }\qquad \displaystyle (N-1) >\frac{p(p+3)}{p-1}\,.\end{cases}$$

 If $f(\cdot)$ satisfies ($f_2$), then it follows that $u=0$ for any\footnote{In the case $f(0)>0$, the solution does not exist at all.} $p>2$.
\end{theorem}

The paper is organized as follows: for the reader's convenience in Section \ref{schemeproofwww} we give a scheme of the proofs.
 We collect some preliminary results in Section \ref{preliminaries}, Section \ref{section3} and Section \ref{section4}.
In Section \ref{th:wcpstripsect} we prove our main result Theorem \ref{th:wcpstrip}.  In  Section \ref{sec666} we provide the  proof of Theorem \ref{mainthm}, Theorem \ref{mainthmfdfdfdfdf} and  Theorem \ref{liouvillenextgenerationtris}.

\section{Scheme of the proofs}\label{schemeproofwww}
\begin{itemize}
\item[(i)] The proof of Theorem \ref{th:wcpstrip} is quite long and somehow technical, we will divide it in various steps.\\
The main idea is to compare $u$ and $v$ on compact sets, and then pass to the limit in  the whole strip.
The limiting process will be carried by a refined iteration technique.\\
\noindent Let us emphasize that an important ingredient is the use of  a weighted Poincar\'{e} inequality (see Section \ref{section4}), that holds true under the abstract assumption \eqref{condcrucial} on $v$. The reader should  keep in mind that, when exploiting Theorem \ref{th:wcpstrip}
to apply the moving plane method, $v$ will be replaced by the reflection of the solution $u$. Namely, in the set $\{y\leq \beta\}$, we will consider $v(x',y)\,:=\,u(x',2\beta-y)$ and therefore we will need to show that actually \eqref{condcrucial} holds true in this case.
This motivates the assumption \eqref{condcrucial}, since we will  prove that it holds  in the case  $v(x',y)\,:=\,u(x',2\beta-y)$.

\noindent \emph{It will be clear from the proof that, modifying \eqref{condcrucial}, we could prove the result for any $p>2$}. \\

\item[(ii)] Taking into account (i), we are lead to prove properties of the summability of $\frac{1}{|\nabla \,u |}$, say in the strip
$\{\beta\leq y\leq 2\beta\}$ (we consider at this stage this case which is the more difficult one). This in fact correspond to proving the summability of $\frac{1}{|\nabla \,v |}$  in the strip
$\{0\leq y\leq \beta\}$. Note that if $u$ approaches zero (that does not occur far from the boundary, in the case of bounded domains), then the estimates we get blow up, and we have to estimate the way this happens. Recall in fact that possibly $u$ may decay at zero also far from the boundary \\

\item[(iii)] We study the summability of $\frac{1}{|\nabla \,u |}$ in Proposition \ref{pr:gradbis}, where we also exploit Proposition \ref{pr:grad}. Actually in the proof of Proposition \ref{pr:gradbis} we use Proposition \ref{pro:SobConstant} that is a quantitative version of some known results in \cite{DS1}.\\

\noindent There is a technical difficulty given by the fact that the constants given by Proposition \ref{pro:SobConstant} depend on the distance of the domain from the boundary, that is an obstruction when $\beta$ small, namely when we will start the moving plane procedure. To overcame this difficulty we will use some scaling arguments.\\

\item[(iv)] The weighted Sobolev (and Poincar\'{e}) inequality are recalled in  Section \ref{section4}, where we also provide a version of it
that allows to split the domain in two parts, and use a different weight function in each sub-domain. This is needed in the proof of Theorem \ref{th:wcpstrip}.\\

\item[(v)] In the \emph{super-linear} case, the proofs are simpler, since it is possible to use some standard translation arguments to show that the solution is monotone near the boundary, and strictly positive (bounded away from zero) far from the boundary. This is proved in
Lemma \ref{le:cuccurucucu} and Lemma \ref{vdcvs0987654}, and then used to prove Theorem \ref{mainthmfdfdfdfdf}.\\

\item[(vi)] The Liouville type results proved in Theorem \ref{liouvillenextgenerationtris} follows by \cite{DFSV}, once we know that the solution
is monotone, and consequently stable.\\
\noindent If the solution $u$ is also bounded, we can study the limiting profile of $u$ at infinity,
 \begin{equation}\nonumber\begin{split}
& w(x')
:=\lim_{t\rightarrow \infty} u(x',y+t).
\end{split}\end{equation}
 which is a stable solution in $\mathbb{R}^{N-1}$, and we get non-existence below the largest critical exponent $q_c(N-1,p)$.
\end{itemize}

\section{Some useful known results}\label{preliminaries}

\

We start stating  a lemma that will be useful in the proof of Theorem \ref{th:wcpstrip}. For the proof we refer the reader to \cite{FMS}.
\begin{lemma}\label{Le:L(R)}
Let $\theta >0$ and $\nu>0$ such that $\theta < 2^{-\nu}$.
Moreover let $R_0>0$, $c>0$ and
$$\mathcal{L}:(R_0, + \infty) \rightarrow \mathbb{R}$$
a non-negative and non-decreasing function such that
\begin{equation}\label{eq:L}
\begin{cases}
\mathcal{L}(R)\leq \theta \mathcal{L}(2R)+g(R) & \forall R>R_0,\\
\mathcal{L}(R)\leq CR^{\nu} & \forall R >R_0,
\end{cases}
\end{equation}
where $g:(R_0, +\infty)\rightarrow \mathbb{R}^+$ is such that
$$\lim_{R\rightarrow +\infty}g(R)=0 .$$ Then
$$\mathcal{L}(R)=0.$$

\end{lemma}

\noindent
Referring to   \cite{V} for the case of the $p$-Laplace operator, and to  \cite{PSB} for the case of a broad class of quasilinear elliptic operators, we recall the following:
\begin{theorem}(Strong Maximum Principle and Hopf's Lemma).\label{semihop}
 Let $\Omega$ be a domain in $\mathbb{R}^N$ and suppose that $u \in C^1(\Omega)$, $u \geqslant 0$
 in $\Omega$, weakly solves
 \[
 -\Delta_p u+cu^q=g \geqslant 0 \quad \mbox{in  }\quad \Omega\,,
 \]
 with $1 < p < \infty$, $q \geqslant p-1$, $c \geqslant 0$ and $g \in L^\infty_{loc}(\Omega)$. If
 $u \neq 0$ then $u >0$ in $\Omega$. Moreover for any point $x_0 \in \partial \Omega$ where the
 interior sphere condition is satisfied, and such that $u \in C^1(\Omega \cup \{x_0\})$ and
 $u(x_0)=0$ we have that $\frac{\partial u}{\partial s}>0$ for any inward directional derivative
 (this means that if $y$ approaches $x_0$ in a ball $B \subseteq \Omega$ that has $x_0$ on its
 boundary, then $\lim_{y \rightarrow x_0}\frac{u(y)-u(x_0)}{|y-x_0|}>0$).
\end{theorem}
\noindent Also we will make repeated use of the following strong comparison principle (see \cite{DS2}):
\begin{theorem}[Strong Comparison Principle]\label{hthCOMPARISONII}
Let $u,v\in C^1(\overline{\Omega})$ where $\Omega$ is a bounded smooth domain  of $\mathbb{R}^N$ with
$\frac{2N+2}{N+2}<p<\infty$. Suppose that either $u$ or $v$ is a weak solution of  $-\Delta _p (w)=f(w)$ with  $f$
positive ($f(s)>0$ for $s>0$) and locally Lipschitz continuous. Assume
\begin{equation}\label{hthCOMP:LAMB22II}
-\Delta_p (u)-f(u) \leqslant -\Delta_p (v)-f(v)\quad\quad u\leqslant v\quad\mbox{in}\quad\Omega\,.
\end{equation}
 Then $u\equiv v$ in $\Omega$ or $u<v$ in $\Omega$.
\end{theorem}
\noindent Let us recall that the linearized operator  $L_u(v,\varphi)$ at a fixed solution $u$ of $-\Delta_p (u)=f(u)$ is well
defined, for every  $v\, ,\, \varphi\in H^{1,2}_{\rho}(\Omega)$  with $\rho\equiv |\nabla u|^{p-2}$(see \cite{DS1} for details), by \ \ \
$$
L_u(v,\varphi) \equiv\int_{\Omega}[|\nabla u|^{p-2}(\nabla v,\nabla \varphi)+(p-2)|\nabla u|^{p-4}(\nabla u,\nabla v)(\nabla u,\nabla \varphi) - f'(u)v\varphi]dx\,.
$$
Moreover, $v\in H^{1,2}_{\rho}(\Omega)$  is a weak solution of the linearized equation if
\begin{equation}\label{EQ:LIN}
L_u(v,\varphi)=0\,,
\end{equation}
for any $\varphi\in H^{1,2}_{0,\rho}(\Omega)$.\\
\noindent By \cite{DS1} we have  $u_{x_i}\in H^{1,2}_{\rho}(\Omega)$
 for $i=1,\ldots ,N$, and
$L_u(u_{x_i},\varphi)$ is well defined
 for every $\varphi\in  H^{1,2}_{0,\rho}(\Omega)$, with
\begin{equation}\label{hthLI:VI}
L_u(u_{x_i},\varphi) =0\quad\quad\forall \varphi\in H^{1,2}_{0,\rho}(\Omega).
\end{equation}
In other words, the derivatives of $u$ are weak solutions of the linearized equation. Consequently by a strong maximum principle for the linearized operator (see \cite{DS2}) we have the following:
\begin{theorem}\label{hthPMFderrr}
Let  $u\in C^1(\overline{\Omega})$  be a weak solution of $-\Delta_p (u)=f(u)$ in a bounded smooth domain $\Omega$ of
$\mathbb{R}^N$ with $\frac{2N+2}{N+2}<p<\infty$, and $f$ positive and locally Lipschitz continuous. Then, for any $i \in \{1,
\dots ,N \}$ and any domain $\Omega '\subset\Omega$ with $ u_{x_{i}}\geqslant 0$ in $\Omega '$, we have
either $u_{x_{i}} \equiv 0$ in $\Omega '$ or $u_{x_{i}} >0$ in $\Omega '$.
\end{theorem}

\section{Preliminary results}\label{section3}
In this paragraph we shall prove some useful results that we need in the proof of the main result. Let us start stating the following:
\begin{basic}\label{rem:GT}
We say that  $u(x)$ satisfies the Condition (PE) in $\Omega$, if
 \begin{equation}\label{eq:hip}
 |u(x)|\leq \hat C \int_\Omega \frac{|\nabla u(y)|}{|x-y|^{N-1}}dy.
\end{equation}
This generally follows by potential estimates, see \cite[Lemma 7.14, Lemma 7.16]{GT}, that gives
\[
u(x)=\hat C\int_{\Omega}\frac{(x_i-y_i)\frac{\partial u}{\partial x_i}(y)}{|x-y|^N}dy\quad\text{a.e.}\,\,(\Omega),
\]
with
\begin{itemize}
\item[$(i)$] $\hat C= \frac{1}{N \omega_N}$ if $u\in W_0^{1,1}(\Omega)$,\\
where $\omega_N$ is the volume of the unit ball in $\mathbb{R}^N$;
\item[$(ii)$] $\hat C= \frac{d^N}{N|S|}$ if $u\in W^{1,1}(\Omega)$ with ${\displaystyle \int_S u=0}$ and $\Omega$ convex, \\
where $d=\text{diam } \Omega$ and $S$ any measurable subset of $\Omega$.
\end{itemize}

\end{basic}
\

Moreover let  $\mu \in (0,1]$, we define
\begin{equation}\label{eq:V}V_\mu[f,U](x)=\int_{U} \frac{f(y)}{|x-y|^{N(1-\mu)}}dy.\end{equation}
  It is well known that (see \cite[pag.159]{GT})
\begin{equation}\label{eq:gilbard}
V_\mu[1,U](x)\leq \mu^{-1}\omega_N^{1-\mu}|U|^\mu.
\end{equation}
Let us state the following:
\begin{lemma}\label{le:potential}
Let us consider $\tilde{\Omega}\subset{\Omega}$ and $V_\mu[f,\Omega](x)$ as in
\eqref{eq:V}. Then for any $1\leq q \leq \infty $ one has
\begin{equation}\label{eq:potential}
||V_\mu[f,\tilde{\Omega}](x)||_{L^q(\Omega)}\leq\left(\frac{1-\delta}{\mu-\delta}\right)^{1-\delta} \omega_n^{1-\mu}|\Omega|^{\mu- \delta}||f||_{L^m(\tilde{\Omega})},
\end{equation}
with $0\leq \delta= {\displaystyle \frac1m - \frac1q}<\mu.$
\end{lemma}
\begin{proof}
The proof follows  by \cite[Lemma 7.12]{GT}.
\end{proof}

Let us recall from \cite{DS1,DCS} the following:
\begin{proposition}\label{pro:SobConstant}
Let $1<p<\infty$ and  $u\in C^{1,\alpha}(\overline{\Omega})$ a solution to
$$
\begin{cases}
-\Delta_p u=h(x), & \text{ in }\Omega\\
u(x) > 0, & \text{ in }\Omega\\
u(x)=0.&  \text{ on }\partial\Omega
\end{cases}
$$
with $h\in C^1 (\Omega)$.
 Let $\Omega'\subset \subset \Omega$ and $0<\delta <dist(\Omega',\partial\Omega)$ such that $h>0$ in $\overline{\Omega'_\delta}$, where
$$\Omega'_\delta=\{x\in \Omega : d(x,\Omega')< \delta\}\subset \subset \Omega.$$
Consider the finite covering  $\Omega'\subseteq \underset{i=1}{\overset{S}{\cup}}B_\delta(x_i)$  with $x_i\in\Omega'$ and $S=S(\delta)$.
We set

\begin{equation}\label{Notationbar}
\begin{split}
&\overline{M}=\max\,\,\{\underset{y\in\Omega'_\delta}{\sup}\int_{\Omega'_\delta}\frac{1}{|x-y|^\gamma}\,dx\, ;\,\underset{y\in\Omega'_\delta}{\sup}\int_{\Omega'_\delta}\frac{1}{|x-y|^{\gamma+1}}\,dx\, ;\, \underset{y\in\Omega'_\delta}{\sup}\int_{\Omega'_\delta}\frac{1}{|x-y|^{\gamma +2}}\,dx\,\,\},\\
&\overline{K}=\underset{x\in \Omega'_\delta}{\sup} |\nabla u(x)|<\infty,\\
&\overline{W}=\underset{x\in \Omega'_\delta}{\sup} |\nabla h (x)|,\\
\end{split}
\end{equation}
for $\gamma< N-2$ if $N\geq 3$, or $\gamma=0$ if $N=2$.
Also let
\begin{equation}
\begin{split}
 &  0<\overline{\mu}\leqslant \underset{x\in \Omega'_\delta}{\inf}\, h(x).\\
 \end{split}
\end{equation}

Then we get:
\begin{equation}\label{drdrdbisssete}
\begin{split}
&\int_{\Omega'} \frac{1}{|\nabla u|^{\tau}}\frac{1}{|x-y|^\gamma}\leqslant \overline{\mathcal{C}}^*\\
\end{split}
\end{equation}
with $\max\{(p-2)\, ,\,0\}\leqslant \tau<p-1$ and

\begin{equation}
\begin{split}
&\overline{\mathcal{C}}^*=\overline{\mathcal{C}}^*(\bar \mu\, ,\,p\, ,\,\gamma\, ,\, \tau\, ,\, f\, ,\, \|u\|_\infty\, ,\, \|\nabla u\|_\infty\, ,\,\delta\,,\,N)=
\\
&\frac{2\,S}{\overline{\mu}}\Bigg[
\frac{N^2\tau^2\,\cdot\overline{M}\cdot \max\{(p-1)\, ,\, 1\}^2 }{\overline\mu(p-\tau-1)} \left ( \frac{\gamma ^2\cdot \overline{K}^{2p-2-\tau}}{p-\tau-1}+\frac{4\overline{K}^{2p-2-\tau}}{(p-\tau-1)\delta^2}+
\frac{\overline{K}^{p-\tau-1}}{p-1}  \overline{W} \right )
\\
&+\gamma \overline{K}^{p-1-\tau}\cdot \overline{M}+\frac{2}{\delta} \cdot \overline{K}^{p-1-\tau}\cdot \overline{M}\Bigg].
\end{split}
\end{equation}
\end{proposition}


\begin{corollary}\label{cor:SobConstant}
With the same hypotheses of Proposition \ref{pro:SobConstant}, assume that  ($f_1$) holds and that
$|\nabla u|$ is bounded. Then
\begin{equation}\label{eq:SobConstant}
\overline{\mathcal{C}}^*\leq C\,\frac{S(\delta)}{\delta^2}\frac{1}{{\displaystyle{a^2\big(\inf_{\Omega'_\delta} u\big)^{2q}}}}\,.
\end{equation}
If else ($f_2$) holds, setting $\sigma =\sigma (\|u\|_\infty)$, we get
\begin{equation}\label{eq:SobConstantbissettretre}
\overline{\mathcal{C}}^*\leq C\,\frac{S(\delta)}{\delta^2}\frac{1}{\sigma^{2q}}\,.
\end{equation}

\end{corollary}
\begin{proof}
It is sufficient to apply Proposition \ref{pro:SobConstant}.
\end{proof}

\begin{remark}\label{countingSdelta}
Later we will exploit Corollary \ref{cor:SobConstant} in the case when the domain is a cube in $\mathbb{R}^N$ of diameter $d$. In this case, the reader may easy deduce that
\[
S(\delta)\leq C\left(\big[\,\frac{d}{\delta}\,\big]+1\right)^N
\]
for some constant  $C$ that only depends on the dimension of  $\mathbb{R}^N$.
\end{remark}

Later we will take advantage of Corollary \ref{cor:SobConstant} in some cases. We will anyway need some refined estimates
on the constant   $\overline{\mathcal{C}}^*$ in the case when ($f_1$) holds. To do this we start with the following:


\begin{proposition}\label{pr:grad}
Let $u\in C^{1,\alpha}$ be a  solution to \eqref{E:P},
and assume that  $|\nabla u|$ is bounded.\\
\noindent Consider  $0<\beta<\beta_0$ and denote
$$ \Sigma_{(\beta,2\beta)}=\left\{ (x',y): x'\in \mathbb{R}^{N-1}, y\in[\beta,2\beta]\right\}\,.$$
If ($f_1$) or ($f_2$) is satisfied, then it follows that there exists a positive constant $C=C(\beta_0 )$
such that
$$|\nabla u|\leq \frac{C}{\beta}\,u,\qquad \forall x\in \Sigma_{\beta-\frac{\beta}{4},\beta+\frac{\beta}{4}}. $$
\end{proposition}
\begin{proof}

\noindent Let us assume that ($f_1$) holds and  let $x'_0$ be fixed and $\mathcal{K}(x'_0)$ be defined by
\[
\mathcal {K}(x'_0)= B_{\beta\sqrt{N}}(x'_0) \times (\beta, 2\beta).
\]
and $B_{\beta\sqrt{N}}(x'_0)$ is the ball in $\mathbb{R}^{N-1}$ of radius $\beta\sqrt{N}$ around $x'_0$.
 Also let us set $u_0= u(x'_0,\frac{3}{2}\beta)$ and
\begin{equation}\label{ghhghgjfjfj}
w_\beta^0(x',y)=\frac{ u((\beta\,x'+x'_0),\beta y)}{u_0} \,.
\end{equation}
Moreover let  $\mathcal{K}_T(x'_0)$ to be the set corresponding to $\mathcal{K}(x'_0)$ via the map:
\[
(x'\,,\,y)\longrightarrow T(x',y)\,:=\, (\beta \,x'+x'_0,\beta\,y)\,,
\]
 that is, $\mathcal{K}_T(x'_0)\,:=\,T^{-1}(\mathcal{K}(x'_0))$. \\
 \noindent In the following we will use the fact that in  domains like $\mathcal{K}(x'_0)$ and  $\mathcal{K}_T(x'_0)$ (or in some their neighborhood) the Harnack constant in the the Harnack inequality
 ( see~\cite[Theorem 7.2.2]{PSB}) is uniformly bounded (not depending on $\beta$).

It follows by ($f_1$) that
\begin{equation}\label{eq:grad:u1}
-\Delta_p\,w_\beta^0=\beta ^p\frac{f(u( (\beta x'+x'_0),\beta y))}{u_0^{p-1}}=d(x)(w_\beta^0)^{p-1} \quad \text{in}\,\,\, \mathcal{K}_T^\mu(x'_0)\,,
\end{equation}
being $\mathcal{K}_T^\mu(x'_0)$ a neighborhood of radius $\mu>0$ of $\mathcal{K}_T(x'_0)$. We may consider e.g. $\mu=\frac{1}{2}$.
Since  $q\geq p-1$, $d(x)$ is uniformly bounded in $\mathcal{K}_T^\mu(x'_0)$ and  we can  exploit Harnack inequality
to get
$$\sup_{\mathcal{K}_T^\mu(x'_0) } w_\beta^0(x',y) \leq C_H \inf_{\mathcal{K}_T^\mu(x'_0)}w_\beta^0(x',y)\leq C_H ,$$
where we also used \eqref{ghhghgjfjfj}.

\

\noindent We can therefore exploit the interior gradient estimates  in \cite{Di}, see in particular Theorem~1 in \cite{Di}, to  get
\[
|\nabla w_\beta^0|\leq C\, \quad \text{in}\,\,\, \mathcal{K}^{\frac{1}{4}}_T(x'_0)\,.
\]
Note that the distance from the boundary of the set $\mathcal{K}_T(x'_0)$ is fixed by construction. Scaling back we get
\[
|\nabla u|\leq \frac{C}{\beta}\, u_0\leq \frac{C\,C_H}{\beta} u\quad \text{in}\,\,\, T(\mathcal{K}^{\frac{1}{4}}_T(x'_0))\,,
\]
using again Harnack inequality. The thesis follows now recalling that  $x_0$ was arbitrary chosen.

\end{proof}

\begin{proposition}\label{pr:gradbis}
Let $u\in C^{1,\alpha}$ be a  solution to \eqref{E:P},
and assume that  $|\nabla u|$ is bounded.\\
\noindent Consider  $0<\beta<\beta_0$ and let $ \Sigma_{(\beta,2\beta)}$ defined as above.
If ($f_1$)  is satisfied, then there exists a constant $C=C(\beta_0)>0$ such that, for any $x_0'\in\mathbb{R}^{N-1}$, it follows
\begin{equation}\label{drdrdbisssetetritthuhueses}
\begin{split}
&\int_{\mathcal{K}(x'_0)} \frac{1}{|\nabla u|^{\tau}}\frac{1}{|x-y|^\gamma}\leq C\,\beta^{(N+\tau-2p-\gamma)}\,u_0^{2p-2q-2-\tau}
\end{split}
\end{equation}
where $u_0= u(x'_0,\frac{3}{2}\beta)$ and
 $\mathcal{K}(x'_0)$ is defined by $ \mathcal {K}(x'_0)= B_{\beta\sqrt{N}}(x'_0) \times (\beta, 2\beta)$,  $\gamma< N-2$ if $N\geq 3$, or $\gamma=0$ if $N=2$ and $\max\{(p-2)\, ,\,0\}\leqslant \tau<p-1$.
\end{proposition}
\begin{remark}\label{ghhghgjfjfjbisrem}
Let us point out for future use that, the value $u_0= u(x'_0,\frac{3}{2}\beta)$ may be substituted by the value of $u$ at some other point of
$\mathcal{K}(x'_0)$. Because of the Harnack inequality, this only causes a changing of the constant $C$ in \eqref{drdrdbisssetetritthuhueses}.
\end{remark}
\

\begin{proof}
   Let
\begin{equation}\label{ghhghgjfjfjbis}
w_\beta(x',y)=\frac{ u( (\beta\,x'+x'_0),\beta y)}{\beta} \,,
\end{equation}
\noindent and consider, as in Proposition \ref{pr:grad},   $\mathcal{K}_T(x'_0)$ to be the set corresponding to $\mathcal{K}(x'_0)$ via the map:
$
(x'\,,\,y)\longrightarrow T(x',y)\,:=\, (\beta\,x'+x'_0,\beta\,y)$. Note that $w_\beta$ is bounded in $\mathcal{K}_T^\mu(x'_0)$ (the neighborhood of radius $\mu>0$ of $\mathcal{K}_T(x'_0)$) by the mean value theorem, and
 it follows by ($f_1$) that
\begin{equation}\label{eq:grad:u1bis}
-\Delta_p\,w_\beta=\beta f(u( (\beta\,x'+x'_0),\beta y)) \quad \text{in}\,\,\, \mathcal{K}_T^\mu(x'_0)\,.
\end{equation}
We  consider  in particular $\mu=\frac{1}{4}$, so that Proposition \ref{pr:grad} applies (see \eqref{drdrdbisssetetritt}).

Also, setting as above $u_0= u(x'_0,\frac{3}{2}\beta)$,  we can  exploit Harnack inequality,  and get
$$u _0\leq \sup_{T(\mathcal{K}_T^\mu(x'_0)) } u \leq C_H \inf_{T(\mathcal{K}_T^\mu(x'_0))}u \leq C_H \,u_0$$

Then, for $\tau <(p-1)$ (actually we will let $\tau\approx (p-1)$) and
$\gamma <N-2$ (actually we will let $\gamma \approx N-2$),
considering $y'\in \mathcal{K}_T(x'_0)$, by exploiting Proposition \ref{pro:SobConstant} for \eqref{eq:grad:u1bis} (fixing e.g. $\delta=1/8$),  we get
\begin{equation}\label{drdrdbisssetetritt}
\begin{split}
&\int_{\mathcal{K}_T(x'_0)} \frac{1}{|\nabla w_\beta|^{\tau}}\frac{1}{|x-y'|^\gamma}\\
 &\leq\frac{C\,\, }{\beta^2\,a^2\,u_0^{2q}\, }\bigg[\big(\frac{u_ 0}{\beta}\big)^{2p-2-\tau}\,+\, \big(\frac{u_0}{\beta}\big)^{p-\tau}\beta^2\,A\,u_0^{q-1}\bigg]+
\frac{C\,\, }{\beta\,a\,u_0^{q}\, }\big(\frac{u_0}{\beta}\big)^{p-1-\tau}\leq\\
&\leq C\,\frac{u_0^{2p-2-\tau} }{\beta^2\beta ^{2p-2-\tau}u_0^{2q}\, }\,,
\end{split}
\end{equation}
where $C=C(\tau,p,\gamma,f)$.\\
\noindent We have used here Proposition \ref{pro:SobConstant} via Proposition \ref{pr:grad}, and we exploited the fact that
 the Harnack constant is uniformly bounded because of the geometry of the domain, see~\cite[Theorem 7.2.2]{PSB}.\\

Scaling,  it is now easy to see that, for $y\in T^{-1}(\mathcal{K}_T(x'_0))$
\begin{equation}\label{drdrdbisssetetritthuhu}
\begin{split}
\int_{\mathcal{K}(x'_0)} \frac{1}{|\nabla u|^{\tau}}\frac{1}{|x-y|^\gamma}&=\beta^{(N-\gamma)}\int_{\mathcal{K}_T(x'_0)} \frac{1}{|\nabla w_\beta|^{\tau}}\frac{1}{|x-\frac{y}{\beta}|^\gamma}\\
&\,\leqslant C\,\beta^{(N-2-\gamma)}\frac{u_0^{2p-2-\tau} }{\beta ^{2p-2-\tau}u_0^{2q}\, }\,,
\end{split}
\end{equation}
where we used \eqref{drdrdbisssetetritt} with  $y'=\frac{y}{\beta}\in (\mathcal{K}_T(x'_0))$.

\end{proof}

The above results will be crucial when dealing with the proof of our main results in the sub-linear case. We now prove
two lemma, that will be mainly used in the super-linear case. We will use in particular some translation arguments which,
in the semilinear case, go back to  \cite{BCN2,Dan1}.

\begin{lemma}\label{le:cuccurucucu}
Let $u\in C^{1,\alpha}(\mathbb{R}^N_+)$ be a  solution to \eqref{E:P} and assume that  ($H_2$) holds, namely
 \[
 f(s)\geq c_f s^{p-1}\qquad \text{in}\quad [0\,,\,s_0]\,,
 \]
 for some $s_0>0$ and some positive constant $c_f$ and $p>2$. Denote
$$ \Sigma_{(\beta,\infty)}=\left\{ (x',y): x'\in \mathbb{R}^{N-1}, y\in(\beta,+\infty)\right\}\,.$$Then there exits $\beta>0$ and $\delta(\beta)>0$ such that
$$ u(x)>\delta(\beta) \quad \forall x \in \Sigma_{(\beta, \infty)}. $$
\end{lemma}
\begin{proof}
Let $\phi^R_1\in C^{1,\alpha}(\overline{B_R(0)})$ be the first positive eigenfunction of the $p$-Laplacian in $B_R(0)$, namely a positive solution of
$$
\begin{cases}\label{eq:1_autovalore}
-\Delta_p \phi =\lambda_1(R)\phi^{p-1}, & \text{ in }B_R(0)\\
\qquad \phi=0,&  \text{ on }\partial B_R(0)\,.
\end{cases}
$$
It is well known that  $\lambda_1(R)\rightarrow 0$ if $R\rightarrow \infty.$ Let us set
$$ \phi^R_{1,x_0}(x)=\phi_1^R(x-x_0)\quad \text{in } B_R(x_0). $$ For  $R$ fixed sufficiently large we have
\begin{eqnarray}\label{eq:cucu1}
-\Delta_p \phi^R_{1,x_0}&=&\lambda_1(R)\left(\phi^R_{1,x_0}\right)^{p-1}\\\nonumber
&\leq&c_f\left(\phi^R_{1,x_0}\right)^{p-1} \quad \text{in } B_R(x_0),
\end{eqnarray}
where $c_f$ is as in the statement. Also by assumption
\begin{equation}\label{eq:cucu2}
-\Delta_p u\geq c_f u^{p-1} \quad \text{in } \mathbb{R}^N_+ \cap \{0\leq u(x) \leq s_0\}.
 \end{equation}
 Moreover we can redefine the first eigenfunction by scaling so that:
 $$\tilde{\phi}^R_{1,x_0}=s \phi^R_{1,x_0}< u \quad \text{in } B_R(x_0).$$

  \noindent Setting $$x_0(i,t)=(x_0'(i,t),y_0(i,t))=x_0+te_i\,,$$ under the condition  $y_0(i,t)>R$,   we can exploit a sliding technique by considering
  \begin{center}
  $\tilde{\phi}^R_{1,i,t}=\tilde{\phi}^R_{1}(x-x_0(i,t)) $,
  \end{center}
  where  $e_i \in S^{N-1}$. By \eqref{eq:cucu1} and \eqref{eq:cucu2}, exploiting  the Strong Comparison Principle (see Theorem \eqref{hthCOMPARISONII}), we get $\tilde{\phi}^R_{1,i,t}< u$
  for every $i,t$, such that $y_0(i,t)>R$.
  In fact, if this is not the case, we would have  $\tilde{\phi}^R_{1,i,t}$ touching from below $u$ at some point, namely
  it would exist some point $\hat x\in \mathbb{R}^N_+$ where $\tilde{\phi}^R_{1,i,t}(\hat x)= u(\hat x)$ for some $i,t$, and $\tilde{\phi}^R_{1,i,t}(\hat x)\leq u(\hat x)$. This (by the Strong Comparison Principle ) would imply $\tilde{\phi}^R_{1,i,t} \equiv u$, that is a contradiction since  $\tilde{\phi}^R_{1,i,t}$ is compactly supported. Thus we have the thesis with $\beta=R$ and $\delta (\beta)= \max \tilde{\phi}^R_{1,x_0}$ in $B_R(x_0)$.
\end{proof}

\begin{lemma}\label{vdcvs0987654}
Let $u\in C^{1,\alpha}(\overline{\mathbb{R}^N})$ be a  solution to \eqref{E:P} such that  $|\nabla u|$ is bounded, with $f$ positive and locally Lipschitz continuous. Assume that
\begin{equation}\label{gfgfgsdhjsjbvbbvbhcnncndnnd}
u\geq \underline{u}_\beta>0\qquad \text{in}\quad \{y\geq\beta\}\,,
\end{equation}
for some $\beta>0$ and some positive constant $\underline{u}_\beta\in \mathbb{R}^+$. Then it follows that
\begin{equation}
\frac{\partial u}{\partial y}\geq \underline{u'}_\theta>0\qquad \text{in}\quad \Sigma_{(0,\theta)}\,,
\end{equation}
for some $\theta >0$ and some positive constant $\underline{u'}_\theta\in \mathbb{R}^+$, with $ \Sigma_{(0,\theta)}=\left\{ (x',y): x'\in \mathbb{R}^{N-1}, y\in[0,\theta]\right\}$.
\end{lemma}

\begin{proof}
We argue by contradiction. Were the claim false, we could find a sequence of points $x_n=(x'_{n},y_{n})$ such that,
for $n$ tending to infinity, we have
 \[
 \frac{\partial u}{\partial y}(x'_n,y_n)\rightarrow 0\quad \text{and}\quad y_n\rightarrow 0\,.
 \]
Let us now define
$$u_n(x',y)=u(x'+x'_n,y)$$
so that $\|u_n\|_\infty = \|u\|_\infty\leqslant C$.  Arguing as in \cite{FMS} and exploiting  Ascoli's theorem,
it follows  that, up to subsequences, we have

\begin{equation}\label{roccaseccatrubis}
u_n\overset{C^{1,\alpha'}_{loc}(\mathbb{R}^N_+)}{\longrightarrow}\tu
\end{equation}
up to subsequences, for for some $\alpha '>0$.
 We consider $\tu$  in the entire space $\mathbb{R}^N_+$ constructed by a standard diagonal process.\\

It is now standard to see that $ -\Delta_p  \tu=f(\tu)$ in $\mathbb{R}^N_+$ and it follows by construction that $\tu\geqslant 0 $ in $\mathbb{R}^N_+$ and consequently  $\tu> 0 $ in $\mathbb{R}^N_+$  by the Strong Maximum Principle (see \cite{PSB,V}),
since the case $\tu\equiv 0 $ is avoided by \eqref{gfgfgsdhjsjbvbbvbhcnncndnnd}. By construction (since $\frac{\partial u}{\partial y}(x'_n,y_n)\rightarrow 0$) it also follows that
\[
\frac{\partial \tu}{\partial y}(0,0)=0\,.
\]
Since $\tu$ is positive in the interior of $\mathbb{R}^N_+$,
the contradiction follows by the Hopf boundary lemma \cite{V},
and the thesis is proved.

\end{proof}

\begin{remark}\label{vdcvs0987654b}
With the same notation of Lemma \ref{vdcvs0987654}, it is now easy to see that actually we may assume that:
\begin{equation}\label{gfgfgsdhjsjbvbbvbhcnncndnndtrtrt}
u\geq \underline{u}_\beta>0\qquad \text{in}\quad \{y\geq\beta\}\,,
\end{equation}
for some $\beta>0$ and some positive constant $\underline{u}_\beta\in \mathbb{R}^+$ and
\begin{equation}
\frac{\partial u}{\partial y}\geq \underline{u'}_\beta>0\qquad \text{in}\quad \Sigma_{(0,2\beta)},
\end{equation}
for some positive constant $\underline{u'}_\beta\in \mathbb{R}^+$.

To prove this it is sufficient to use the monotonicity of the solution near the boundary, given by Lemma \ref{vdcvs0987654}, and
the Harnack inequality (see~\cite[Theorem 7.2.2]{PSB}) far from the boundary, like in Proposition \ref{pr:gradbis}.
\end{remark}

\section{A weighted Sobolev-type inequality}\label{section4}
\begin{theorem}\label{thm: Sobolev}
Consider two sets $\Omega_1$ and $\Omega_2$ such that $\Omega_1\subset\Omega$, $\Omega_2 \subset\Omega$, $\Omega_1 \cap\Omega_2= \emptyset$ and $\overline{\Omega_1 \cup\Omega_2}= \overline{\Omega}$.

Let $\rho$ and $\eta$ two weight functions such that
\begin{eqnarray}\label{eq:weight}
\int_{\Omega_1} \frac{1}{\rho^t|x-y|^{\gamma}}\leq C^*_1, \\\nonumber
\int_{\Omega_2} \frac{1}{\eta^t|x-y|^{\gamma}}\leq C^*_2,
\end{eqnarray}
with $t=\frac{p-1}{p-2}r$, $\frac{p-2}{p-1}<r<1$, $\gamma< N-2$ $(\gamma =0 \,\,\text{if}\,\, N=2)$. Assume, in the case $N\geq 3$, without no lose of generality that
$$\gamma>N-2t,$$
which\footnote{Note that the condition $\gamma > N\,-\,2t$ holds true for $r\approx 1$ and $\gamma \approx N-2$ that we may assume with no loose of generality. } implies $Nt-2N+2t+\gamma>0$. Then,  for any $w\in H^{1,2}_{0,\rho}(\Omega_1)\cap H^{1,2}_{0,\eta}(\Omega_2)$,
 there exist  constants $C_{s_\rho}$ and $C_{s_\eta}$ such that
\begin{eqnarray}\label{FTFTnddjncj}
||w||_{L^q(\Omega)}&\leq& C_{s_\rho}||\nabla w||_{L^2(\Omega_1,\rho)}+C_{s_\eta}||\nabla w||_{L^2(\Omega_2,\eta)} \\\nonumber
&=&C_{s_\rho}\left( \int_{\Omega_1}\rho|\nabla w|^2\right)^{\frac 12}+C_{s_\eta}\left( \int_{\Omega_2}\eta|\nabla w|^2\right)^{\frac 12},
\end{eqnarray}
for any $1\leq q< 2^*(t)$
where
\begin{equation}\label{eq:2*}
\frac{1}{2^*(t)}=\frac 12 - \frac 1N + \frac 1t \left( \frac 12 - \frac{\gamma}{2N}\right).
\end{equation}
with
\begin{equation}\label{eq:Cs}
C_{s_\rho}=\hat C(C^*_1)^{\frac{1}{2t}}(C_M)^\frac{1}{(2t)'}\quad \text{and}\quad C_{s_\eta}=\hat C(C^*_2)^{\frac{1}{2t}}(C_M)^\frac{1}{(2t)'},
\end{equation}
where $\hat C$ is as in Condition (PE) \ref{rem:GT}, $C^*_1$ and  $C^*_2$ are as in the statement of theorem and $$C_M=\left(\frac{1-\delta}{\frac \alpha N-\delta}\right)^{1-\delta} \omega_n^{1-\frac{\alpha}{N}}|\Omega|^{\frac{\alpha}{N}-\delta}.$$
\end{theorem}

\begin{remark}\label{g675}
Note that the largest value of $2^*(t)$ is obtained at the limiting case $t\approx \frac{p-1}{p-2}$, and $\gamma \approx (N-2)$,
$\gamma =0$ for $N=2$. We
have therefore that  \eqref{FTFTnddjncj} holds for any $q<\tilde2^*$ where
\[
\frac{1}{ \tilde2^*}=\frac{1}{2}-\frac{1}{N}+\frac{p-2}{p-1}\,\cdot\,\frac{1}{N}\, ,
\]
Moreover one has $\tilde2^*>2$.
\end{remark}
\begin{proof}
Without loss of generality we assume $w$ belonging to  $C^1(\Omega)$ or $C_0^1(\Omega)$ depending on the case $(i)$ or $(ii)$ of Condition (PE). Hence equation \eqref{eq:hip} implies
\begin{equation}\label{eq:hip_real}|w(x)|\leq \hat C \int_{\Omega}\frac{|\nabla w(y)|}{|x-y|^{N-1}}dy.
\end{equation}
Then
\begin{eqnarray}\nonumber
|w(x)|&\leq& \hat C\int_{\Omega_1}\frac{|\nabla w(y)|}{|x-y|^{N-1}}dy+\hat C\int_{\Omega_2}\frac{|\nabla w(y)|}{|x-y|^{N-1}}dy\\\nonumber
&\leq&\hat C\int_{\Omega_1}\frac{1}{\rho^\frac 12|x-y|^{\frac{\gamma}{2t}}}\frac{|\nabla w(y)|\rho^{\frac 12}}{|x-y|^{N-1-\frac{\gamma}{2t}}}dy+\hat C\int_{\Omega_2}\frac{1}{\eta^\frac 12|x-y|^{\frac{\gamma}{2t}}}\frac{|\nabla w(y)|\eta^{\frac 12}}{|x-y|^{N-1-\frac{\gamma}{2t}}}dy\\\nonumber
&\leq&\hat C\left(\int_{\Omega_1}\frac{1}{\rho^t|x-y|^\gamma}dy\right)^{\frac{1}{2t}}\left(\int_{\Omega_1}\frac{\left(|\nabla w(y)|\rho^{\frac 12}\right)^{(2t)'}}{|x-y|^{(N-1-\frac{\gamma}{2t})(2t)'}}dy\right)^{\frac{1}{(2t)'}}\\\nonumber
&+&\hat C\left(\int_{\Omega_2}\frac{1}{\eta^t|x-y|^\gamma}dy\right)^{\frac{1}{2t}}\left(\int_{\Omega_2}\frac{\left(|\nabla w(y)|\eta^{\frac 12}\right)^{(2t)'}}{|x-y|^{(N-1-\frac{\gamma}{2t})(2t)'}}dy\right)^{\frac{1}{(2t)'}},
\end{eqnarray} where  in the last inequality we used H\"older inequality with $\frac{1}{2t}+\frac{1}{(2t)'}=1$. Hence
\begin{eqnarray}\label{eq:w1}
|w(x)|&\leq& \hat C (C^*_1)^{\frac{1}{2t}}\left(\int_{\Omega_1}\frac{\left(|\nabla w(y)|\rho^{\frac 12}\right)^{(2t)'}}{|x-y|^{(N-1-\frac{\gamma}{2t})(2t)'}}dy\right)^{\frac{1}{(2t)'}}\\\nonumber
&+&\hat C (C^*_2)^\frac{1}{2t}\left(\int_{\Omega_2}\frac{\left(|\nabla w(y)|\eta^{\frac 12}\right)^{(2t)'}}{|x-y|^{(N-1-\frac{\gamma}{2t})(2t)'}}dy\right)^{\frac{1}{(2t)'}}.
\end{eqnarray}
We point out that
\begin{equation}\label{eq:f_summability}
(|\nabla w|\rho^{\frac 12})^{(2t)'}\in L^{\frac{2}{(2t)'}}(\Omega_1)\qquad \text{and}\qquad (|\nabla w|\eta^{\frac 12})^{(2t)'}\in L^{\frac{2}{(2t)'}}(\Omega_2).
\end{equation}
From \eqref{eq:w1}, by using equation \eqref{eq:V} with $\mu=1 -\frac{1}{N}(N-1-\frac{\gamma}{2t})(2t)'$, we obtain
\begin{eqnarray}\label{eq:w2}
|w(x)|&\leq& \hat C (C^*_1)^{\frac{1}{2t}}\left(V_\mu\left[\left(|\nabla w(y)|\rho^\frac12\right)^{(2t)'},\Omega_1\right](x)\right)^\frac{1}{(2t)'}\\\nonumber
&+& \hat C (C^*_2)^{\frac{1}{2t}}\left(V_\mu\left[\left(|\nabla w(y)|\eta^\frac12\right)^{(2t)'},\Omega_2\right](x)\right)^\frac{1}{(2t)'}.
\end{eqnarray}
Moreover we remark that the assumption $\gamma>N-2t$ implies $\mu>0$.

\

We shall use now Lemma \ref{le:potential} setting
$$\displaystyle \frac 1m=\frac{(2t)'}{2},$$ see~\eqref{eq:f_summability}.  In order to apply \eqref{eq:potential}, since by assumption $Nt-2N+2t+\gamma>0$,   a direct calculation shows  that it is possible to find a $q>1$ such that
$$\frac 1m - \frac 1q <\mu.
$$
From \eqref{eq:w1} we have
\begin{eqnarray}\nonumber\\\nonumber
\left( \int_{\Omega}|w(x)|^{q(2t)'}dx\right)^{\frac{1}{q(2t)'}}&\leq&
 \Bigg(\int_{\Omega} \bigg(\hat C (C^*_1)^{\frac{1}{2t}}\Big(V_\mu\left[\left(|\nabla w(y)|\rho^\frac12\right)^{(2t)'},\Omega_1\right](x)\Big)^\frac{1}{(2t)'}\\\nonumber &+&\hat C (C^*_2)^{\frac{1}{2t}}\Big(V_\mu\left[\left(|\nabla w(y)|\eta^\frac12\right)^{(2t)'},\Omega_2\right](x)\Big)^\frac{1}{(2t)'}\bigg)^{q(2t)'}dx\Bigg)^\frac{1}{q(2t)'}
\end{eqnarray}
and by Minkowski inequality
\begin{eqnarray}\label{eq:ws1}\\
\nonumber
\left( \int_{\Omega}|w(x)|^{q(2t)'}dx\right)^{\frac{1}{q(2t)'}}&\leq&
 \hat C (C^*_1)^{\frac{1}{2t}}\Bigg\|V_\mu\left[\left(|\nabla w(y)|\rho^\frac12\right)^{(2t)'},\Omega_1\right](x)\Bigg\|_{L^q(\Omega)}^\frac{1}{(2t)'}
\\\nonumber &+&
\hat C (C^*_2)^{\frac{1}{2t}}\Bigg\|V_\mu\left[\left(|\nabla w(y)|\eta^\frac12\right)^{(2t)'},\Omega_2\right](x)\Bigg\|_{L^q(\Omega)}^\frac{1}{(2t)'}.
\end{eqnarray}

From \eqref{eq:ws1}, by using  Lemma \ref{le:potential} we get
\begin{eqnarray}\\\nonumber
\left( \int_{\Omega}|w(x)|^{q(2t)'}\right)^{\frac{1}{q(2t)'}}&\leq&
\hat C(C^*_1)^{\frac{1}{2t}}\left(\left(\frac{1-\delta}{\frac \alpha N-\delta}\right)^{1-\delta} \omega_n^{1-\frac{\alpha}{N}}|\Omega|^{\frac{\alpha}{N}-\delta}\right)^\frac{1}{(2t)'}\left(\int_{\Omega_1}\rho|\nabla w|^2\right)^{\frac 12}\\\nonumber &+&\hat C(C^*_2)^{\frac{1}{2t}}\left(\left(\frac{1-\delta}{\frac \alpha N-\delta}\right)^{1-\delta} \omega_n^{1-\frac{\alpha}{N}}|\Omega|^{\frac{\alpha}{N}-\delta}\right)^\frac{1}{(2t)'}\left(\int_{\Omega_2}\eta|\nabla w|^2\right)^{\frac 12},
\end{eqnarray}
that gives  \eqref{eq:2*} with $q(2t)'=2^*(t)$ and \eqref{eq:Cs} with $C_M=\left(\frac{1-\delta}{\frac \alpha N-\delta}\right)^{1-\delta} \omega_n^{1-\frac{\alpha}{N}}|\Omega|^{\frac{\alpha}{N}-\delta}$.
\end{proof}
Now we are ready to state the following
\begin{corollary}[Weighted Poincar\'e inequality]\label{cor:PoincarŽ}
Let $w$ be as in one of the following cases
\begin{itemize}
\item [$(i)$] $w\in H^{1,2}_{0,\rho}(\Omega)\cap H^{1,2}_{0,\eta}(\Omega)$,
\item [$(ii)$] $w\in H^{1,2}_{\rho}(\Omega)\cap H^{1,2}_{\eta}(\Omega)$ such that ${\displaystyle \int_\Omega w=0}$ and $\Omega$ convex,
\end{itemize}
and  $\Omega_1$ and $\Omega_2$ such that $\Omega_1\subset\Omega$, $\Omega_2 \subset\Omega$, $\Omega_1 \cap\Omega_2= \emptyset$ and $\overline{\Omega_1 \cup\Omega_2}= \overline{\Omega}$.

\

Then, if the weights $\rho$ and $\eta$ fulfill \eqref{eq:weight}, then
$$\int_{\Omega}w^2\leq {C_p}(\Omega)\hat C^2\left((C^*_1)^{\frac{1}{t}}(C_M)^\frac{2}{(2t)'} \int_{\Omega_1}\rho|\nabla w|^2+(C^*_2)^{\frac{1}{t}}(C_M)^\frac{2}{(2t)'} \int_{\Omega_2}\eta|\nabla w|^2\right),$$
where $\hat C, C^*_1, C^*_2, C_M$ are as in Theorem \ref{thm: Sobolev} and with ${C_p}(\Omega)\rightarrow 0$ if $|\Omega|\rightarrow 0$.\\
In particular, given any $0<\theta<1$, we can assume that
\begin{equation}\label{estimacav}
C_p(\Omega)\,\leq\,C\,|\Omega|^{\frac{2\,\theta}{(p-1)N}}\,.
\end{equation}
\end{corollary}
\begin{proof}  Choose $2< q<\tilde 2^*$. By Holder inequality we get:
\begin{equation}\label{poinstettema}
\int_{\Omega}w^2\leq\left(\int_{\Omega}w^q\right)^\frac 2q |\Omega|^\frac{q-2}{q},
\end{equation}
and then  using Theorem \ref{thm: Sobolev} one has
$$\int_{\Omega}w^2\leq {C_p}(\Omega)\hat C^2\left((C^*_1)^{\frac{1}{t}}(C_M)^\frac{2}{(2t)'} \int_{\Omega_1}\rho|\nabla w|^2+(C^*_2)^{\frac{1}{t}}(C_M)^\frac{2}{(2t)'} \int_{\Omega_2}\eta|\nabla w|^2\right).$$
By \eqref{poinstettema} and direct computation it follows \eqref{estimacav}.
\end{proof}
\section{A Weak Comparison Principle in narrow domains, Proof of Theorem \ref{th:wcpstrip}}\label{th:wcpstripsect}

We prove here below Theorem \ref{th:wcpstrip}. Let us start considering the case when ($H_1$) is assumed to hold, that is
($f_1$) holds and $2<p<3$. Since $u$ and $v$ are bounded, in the formulation of ($f_1$)  we fix $\mathcal{M}=\max\{\|u\|_\infty\,;\,\|v\|_\infty\}$
and $a=a(\mathcal{M})>0$ and $A=A(\mathcal{M})>0$ such that
\begin{equation}\label{ftfthuhuhcvcvcdffddfsesees}
\begin{split}
&a\, u^q\leq f(u)\leq A\, u^q\qquad \text{and } \qquad  \big |f'(u)\big |\leq A \,u^{q-1}\\
&a\, v^q\leq f(v)\leq A\, u^q\qquad \text{and } \qquad  \big |f'(v)\big |\leq A \,v^{q-1}
\end{split}
\end{equation}


\noindent In the sequel we further use the following  inequalities: \\

$\forall \eta, \eta' \in  \mathbb{R}^{N}$ with $|\eta|+|\eta'|>0$ there exists positive constants $C_1, C_2$ depending on $p$ such that
\begin{eqnarray}\label{eq:inequalities}
[|\eta|^{p-2}\eta-|\eta'|^{p-2}\eta'][\eta- \eta'] \geq C_1 (|\eta|+|\eta'|)^{p-2}|\eta-\eta'|^2, \\ \nonumber\\\nonumber
||\eta|^{p-2}\eta-|\eta'|^{p-2}\eta '|\leq C_2 (|\eta|+|\eta'|)^{p-2}|\eta-\eta '|,\\\nonumber\\\nonumber
[|\eta|^{p-2}\eta-|\eta'|^{p-2}\eta '][\eta-\eta ']\geq C_3 |\eta-\eta '|^p \qquad\mbox{if}\quad p\geq 2.
\end{eqnarray}
First of all we remark that  $(u-v)^+ \in L^{\infty}(\Sigma_{(\lambda, \beta)})$ since we assumed $u,v$ to be bounded in $\Sigma_{(\lambda, \beta)}$.\\

Let us now define
\begin{equation}\label{Eq:Cut-off}\Psi=[(u-v)^+]^{\alpha} \varphi_R^2,\end{equation} where $\alpha > 1,$ will be fixed later and $\varphi_R(x',y)=\varphi_R(x') \in C^{\infty}_c (\mathbb{R}^{N-1}) $, $\varphi_R \geq 0$ such that
\begin{equation}\label{Eq:Cut-off1}
\begin{cases}
\varphi_R \equiv 1, & \text{ in } B^{'}(0,R) \subset \mathbb{R}^{N-1},\\
\varphi_R \equiv 0, & \text{ in } \mathbb{R}^{N-1} \setminus B^{'}(0,2R),\\
|\nabla \varphi_R | \leq \frac CR, & \text{ in } B^{'}(0, 2R) \setminus B^{'}(0,R) \subset  \mathbb{R}^{N-1},
\end{cases}
\end{equation}
where $B^{'}(0,R)$ denotes the ball in $\mathbb{R}^{N-1}$ with center $0$ and radius $R>0$. From now on, for the sake of simplicity, we set $\varphi_R(x',y):=\varphi(x',y)$.

\

We note that $\Psi \in W_0^{1,p}(\Sigma_{(\lambda, \beta)})$ by \eqref{Eq:Cut-off1} and since $u\leq v \text{ on } \partial \Sigma_{(\lambda, \beta)}$.

Let us define the cylinder $$\mathcal{C}_{(\lambda, \beta)}(R)=\mathcal{C}(R):=\left\{ \Sigma_{(\lambda, \beta)}\cap \overline{\{B^{'}(0,R)\times \mathbb{R}\}} \right\}.$$
Then using $\Psi$ as test function in both equations of problem \eqref{Eq:WCP} and substracting we get
\begin{eqnarray}\label{eq:cn1}\\\nonumber
&&\alpha\int_{\mathcal{C}(2R)}\big(|\nabla u|^{p-2}\nabla u-|\nabla v|^{p-2}\nabla v,\nabla(u-v)^+\big)[(u-v)^+]^{\alpha-1}\varphi^2\\\nonumber&+&\int_{\mathcal{C}(2R)}\big(|\nabla u|^{p-2}\nabla u-|\nabla v|^{p-2}\nabla v,\nabla \varphi^2\big)[(u-v)^+]^{\alpha}\\\nonumber
&= &\int_{\mathcal{C}(2R)}(f(u)-f(v))[(u-v)^+]^{\alpha}\varphi^2.
\end{eqnarray} By \eqref{eq:inequalities} and the fact that $p \geq 2$, from \eqref{eq:cn1} one has
\begin{eqnarray}\label{eq:cn2}
&&\alpha \dot  C\int_{\mathcal{C}(2R)}(|\nabla u|+|\nabla v|)^{p-2}|\nabla(u-v)^+|^2[(u-v)^+]^{\alpha-1}\varphi^2 \\\nonumber &\leq&\alpha\int_{\mathcal{C}(2R)}\big(|\nabla u|^{p-2}\nabla u-|\nabla v|^{p-2}\nabla v,\nabla(u-v)^+\big)[(u-v)^+]^{\alpha-1}\varphi^2\\\nonumber
&=& - \int_{\mathcal{C}(2R)}\big(|\nabla u|^{p-2}\nabla u-|\nabla v|^{p-2}\nabla v,\nabla \varphi^2\big)[(u-v)^+]^{\alpha}\\\nonumber &+&\int_{\mathcal{C}(2R)}(f(u)-f(v))[(u-v)^+]^{\alpha}\varphi^2\\\nonumber &\leq&
\int_{\mathcal{C}(2R)}\left |\big(|\nabla u|^{p-2}\nabla u-|\nabla v|^{p-2}\nabla v,\nabla \varphi^2\big)\right|[(u-v)^+]^{\alpha} \\\nonumber &+&\int_{\mathcal{C}(2R)}(f(u)-f(v))[(u-v)^+]^{\alpha}\varphi^2\\\nonumber
&\leq&\check C\int_{\mathcal{C}(2R)} (|\nabla u|+|\nabla v|)^{p-2}|\nabla(u-v)^+| |\nabla \varphi^2| [(u-v)^+]^{\alpha} \\\nonumber &+&\int_{\mathcal{C}(2R)}(f(u)-f(v))[(u-v)^+]^{\alpha}\varphi^2,\end{eqnarray}
where in the last line we used Schwarz inequality and the second of \eqref{eq:inequalities}.

\

Setting
\begin{equation}\label{eq:I1}{ I_1:=\check C\int_{\mathcal{C}(2R)} (|\nabla u|+|\nabla v|)^{p-2}|\nabla(u-v)^+| |\nabla \varphi^2| [(u-v)^+]^{\alpha} }
\end{equation}
and
\begin{equation}\label{eq:I2'}
I_2:=\int_{\mathcal{C}(2R)}(f(u)-f(v))[(u-v)^+]^{\alpha}\varphi^2,
\end{equation}
equation \eqref{eq:cn2} becomes
\begin{equation}\label{eq:cn33}
\alpha \dot  C\int_{\mathcal{C}(2R)}(|\nabla u|+|\nabla v|)^{p-2}|\nabla(u-v)^+|^2[(u-v)^+]^{\alpha-1}\varphi^2\leq I_1+I_2.
\end{equation}
We proceed in three steps:

\

\emph{ Step 1: Evaluation  of $I_1$.}\\

\noindent From \eqref{eq:I1}, we obtain
\begin{eqnarray}\label{eq:cn3}\\\nonumber
I_1&=&2\check C\int_{\mathcal{C}(2R)} (|\nabla u|+|\nabla v|)^{p-2}|\nabla(u-v)^+| \varphi|\nabla \varphi| [(u-v)^+]^{\alpha} \\\nonumber
&=&2\check C\int_{\mathcal{C}(2R)} (|\nabla u|+|\nabla v|)^{\frac{p-2}{2}}|\nabla(u-v)^+| \varphi [(u-v)^+]^{\frac{\alpha-1}{2}} (|\nabla u|+|\nabla v|)^{\frac{p-2}{2}}|\nabla \varphi|  [(u-v)^+]^{\frac{\alpha+1}{2}}\\\nonumber &\leq&\delta' \check C\int_{\mathcal{C}(2R)}(|\nabla u|+|\nabla v|)^{p-2}|\nabla(u-v)^+|^2
\varphi^2[(u-v)^+]^{\alpha-1}
\\\nonumber
&+&\frac{\check C}{\delta'}\int_{\mathcal{C}(2R)}(|\nabla u|+|\nabla v|)^{p-2}
|\nabla\varphi|^2[(u-v)^+]^{\alpha+1},
\end{eqnarray} where in the last inequality  we used weighted Young inequality, and $\delta'$ will be chosen later.
 Hence
 \begin{equation}\label{eq:supI_1}
 I_1\leq I_1^a+I_1^b,
 \end{equation} where
 \begin{eqnarray}\label{eq:supI_1^a}
 I_1^a&:=&\delta' \check C\int_{\mathcal{C}(2R)}(|\nabla u|+|\nabla v|)^{p-2}|\nabla(u-v)^+|^2
\varphi^2[(u-v)^+]^{\alpha-1},\\\nonumber
I_1^b&:=&\frac{\check C}{\delta'}\int_{\mathcal{C}(2R)}(|\nabla u|+|\nabla v|)^{p-2}
|\nabla\varphi|^2[(u-v)^+]^{\alpha+1}.
 \end{eqnarray}

\

Let us consider now $\overline{N}=\overline{N}(R)$ cubes $Q_i$ with edge $l=\beta - \lambda$ and with  the $y-$coordinate of the center, say $y_C$,   such that  $y_C=\frac{\beta+\lambda}{2}$. More precisely we indicate with $(x_0^i\,,\,\frac{\beta+\lambda}{2})$ the center of the cube $Q_i$.
Moreover we assume that   $Q_i\cap Q_j=\emptyset$ for $i\neq j$ and
\begin{equation}\label{eq:Qunion}
\bigcup_{i=1}^{\overline{N}}\overline{Q_i}\supset \mathcal{C}(2R).
\end{equation}
It follows as well,  that each cube $Q_i$ has diameter
\begin{equation}\label{eq:diameterQ}
\text{diam}(Q_i)=d_Q=\sqrt{N}(\beta-\lambda), \qquad  i=1,\cdots,\overline{N}.
\end{equation}
The idea in considering the union  \eqref{eq:Qunion}, is to use in each cube $Q_i$ the weighted Poincar\'e inequality, see  Corollary \ref{cor:PoincarŽ} and taking advantage of the constant $\hat C$ that turns to be not depending on the index $i$ of \eqref{eq:Qunion}.   In fact let us define
\begin{equation}\label{eq:w}
w(x):=
\begin{cases}
\Big(u-v\Big)^+(x',y) & \text{if } (x',y)\in \overline{Q}_i; \\
-\Big(u-v\Big)^+(x',2\beta-y) & \text{if } (x',y)\in  \overline{Q}_i^r,
\end{cases}
\end{equation}
where $(x',y)\in \overline{Q}_i^r$ iff $(x',2\beta-y)\in \overline{Q}_i$.
\\
Since ${\displaystyle{\int_{Q_i \cup Q_i^r}w(x)dx=0}}$, we have that
$$w(x)=\hat C\int_{{Q_i \cup Q_i^r}}\frac{(x_i-z_i)D_iw(z)}{|x-z|^N}dz\quad  \,\,\text{a.e. } x\in {Q_i \cup Q_i^r},$$
where $\hat C= \frac{(\beta-\lambda)^N}{N|Q_i \cup Q_i^r|}$. Then for almost every $ x\in {Q_i}$ one has
\begin{eqnarray}\label{eq:wchanging}
|w(x)|&\leq& \hat C\int_{{Q_i \cup Q_i^r}}\frac{|\nabla w(z)|}{|x-z|^{N-1}}dz\\\nonumber
&=&\hat C\int_{{Q_i}}\frac{|\nabla w(z)|}{|x-z|^{N-1}}dz+\hat C\int_{{Q_i^r}}\frac{|\nabla w(z)|}{|x-z|^{N-1}}dz\\\nonumber
&\leq& 2\hat C\int_{{Q_i}}\frac{|\nabla w(z)|}{|x-z|^{N-1}}dz\,,
\end{eqnarray}
where in the last line we used, the following  standard changing of variables
$$
\begin{cases}
z_1^t=z_1\\
\vdots
\\
z_{N-1}^t=z_{N-1}\\
z_N^t=2\beta - z_N,
\end{cases}
$$
the fact that for $x\in Q_i$, one has $(|x-z|)\Big |_{z\in Q_i}\leq (|x-z^t|)\Big |_{z\in Q_i}$ and that, by \eqref{eq:w} it holds $|\nabla w(z)|=|\nabla w(z^t)|$.  Once we have \eqref{eq:wchanging}, the proof of Theorem \ref{thm: Sobolev} applies by considering $\Omega=Q_i$,  that is the case we are interested here.

\

\begin{itemize}
\item[-] We analyze the term $I_1^b$.
\end{itemize}

\noindent By \eqref{Eq:Cut-off1} and   since $\nabla u, \nabla v \in L^{\infty}(\Sigma_{(\lambda,y_0)})$, we have
\begin{eqnarray}\label{eq:cn4}
I_1^b\leq\sum_{i=1}^{\overline{N}}\frac{C}{\delta' R^2}\int_{\mathcal{C}(2R) \cap Q_i }\left([(u-v)^+]^{\frac{\alpha+1}{2}}\right)^2.
\end{eqnarray}
Then we are going to  use Corollary  \ref{cor:PoincarŽ} with
$$\Omega_1^i=\mathcal{C}(2R) \cap Q_i \cap\{u > \frac{1}{R^{m}}\}$$ and
$$\Omega_2^i=\mathcal{C}(2R) \cap Q_i \cap\{u \leq \frac{1}{R^{m}}\},$$ 
with $m>0$ to be chosen later and considering the weight $\eta \equiv 1$ in $\Omega_2^i$ and the weight $\rho \equiv |\nabla u|^{p-2}$ in $\Omega_1^i$.\\
\noindent At this stage, it is important to note that actually each domain $\Omega^i_{1}$ and $\Omega^i_{2}$  depends in fact on $R$.  Anyway, to make simpler  the reading,   we use the notation $\Omega^i_{1}$ instead of $\Omega^i_{1}(R)$ and $\Omega^i_{2}$ instead of $\Omega^i_{2}(R)$.\\
\noindent 
We set
\begin{equation}\label{eq:Csharp}
C^{\sharp}=C_p(\Omega_1^i)\cdot\hat C^2(d_Q)\cdot(C^*_1)^{\frac{1}{t}}\cdot(C_M)^\frac{2}{(2t)'},
\end{equation}
where all the constants are those given in Corollary \ref{cor:PoincarŽ}. Let us emphasize that the Poincar\'{e} constant in $\Omega_2^i$
is estimated as for the standard (not-weighted) case, since we choose $\eta\equiv 1$ in  $\Omega_2^i$.

\

Thus,   using  Corollary \ref{cor:PoincarŽ} and the classical inequality $(a+b)^p\leq2^p(a^p+b^p)$, for $a,b>0$,  we get from \eqref{eq:cn4}
\begin{eqnarray}\label{eq:cn5}\\\nonumber
I_1^b&\leq&\sum_{i=1}^{\overline{N}}\Bigg( C^{\sharp}\frac{C}{\delta' R^2}\int_{\Omega_1^i}|\nabla u|^{p-2}\left|\nabla[(u-v)^+]^{\frac{\alpha+1}{2}}\right|^2\\\nonumber
&+&\frac{C(\Omega_2^i,d_Q,\alpha)}{\delta' R^2}\int_{\Omega_2^i}[(u-v)^+]^{\alpha-1}|\nabla(u-v)^+|^2 \Bigg)\\\nonumber
&\leq&\sum_{i=1}^{\overline{N}} C^{\sharp}\frac{C(\alpha,\delta' )}{R^2}\int_{\Omega_1^i}|\nabla u|^{p-2}[(u-v)^+]^{\alpha -1}\left|\nabla[(u-v)^+\right|^2\\\nonumber
&+&2^{N-1}\beta\omega_{N-1} \frac{C(\Omega_2^i,d_Q,\alpha,\delta')}{R^{2+m(\alpha-1)}}R^{N-1},
\end{eqnarray}
being $2^{N-1}{\displaystyle\beta R^{N-1}\omega_{N-1}\geq\sum_{i=1}^{\overline{N}}|\Omega_2^i|}$, where $\omega_{N-1}$ is the volume of the unit ball in $\mathbb{R}^{N-1}$. Thus \eqref{eq:cn5} states as
\begin{eqnarray}\label{eq:cn55}\\\nonumber
I_1^b&\leq&\sum_{i=1}^{\overline{N}}C^{\sharp}\frac{C(\alpha,\delta')}{R^2}\int_{\Omega_1^i}|\nabla u|^{p-2}[(u-v)^+]^{\alpha -1}\left|\nabla[(u-v)^+\right|^2\\\nonumber
&+&\frac{C(\Omega_2^i,d_Q,\alpha,\delta',\beta,N)}{R^{2+m(\alpha -1)+1-N}}.
\end{eqnarray}
To estimate  $C^\sharp$  we are going to estimate the constant $C^*_1$ in   \eqref{eq:Csharp}.

\

Since we are considering the domain $\Omega_1^i=\mathcal{C}(2R) \cap Q_i \cap\{u > \frac{1}{R^{m}}\}$, we have  that
\begin{equation}\label{eq:dist}dist\Big(\Omega^i_1, \{u=0\}\Big) \,\geq\,\frac{1}{\|\nabla \,u\|_\infty \,} \frac{1}{R^m}> \frac{C}{R^m},
\end{equation}
for some positive constant $C$, that does not depend on $R$ since $|\nabla u|$ is bounded. In fact by mean value theorem one has $u(x',y)\leq C y$, that implies \eqref{eq:dist} by the definition of $\Omega_1^i.$

\noindent Now we apply  Corollary \ref{cor:SobConstant} with
 $\delta=\frac{\epsilon}{R^m}$
 and $\epsilon$ fixed sufficiently small in order that
 \[
 u > \frac{1}{2\,R^{m}}
 \]
 in the neighborhood of radius $\delta $ of  $\Omega_1^i$. Note that such $\epsilon>0$ exists, and does not depend on $R$, since the gradient of $u$ is bounded.

\noindent  Moreover the number $S=S(\delta)$ for the covering of every $Q_i$ (see Proposition \ref{pro:SobConstant} and Remark \ref{countingSdelta}) can be estimated by $$S\leq C\,R^{m\,N},$$
 for some constant $C>0$.

\

\noindent Exploiting Corollary \ref{cor:SobConstant} (see \eqref{eq:SobConstant}), we obtain
$$C_1^*\leq C \,R^{(N+2q+2)m} \quad \text{in } \Omega^i_{1}.$$
Thus equation \eqref{eq:cn55}, by using \eqref{eq:Csharp},   becomes
\begin{eqnarray}\label{eq:cn6}\\\nonumber
I_1^b&\leq&C^\flat\int_{\mathcal{C}(2R)}(|\nabla u|+|\nabla v|)^{p-2 }[(u-v)^+]^{\alpha -1}\left|\nabla(u-v)^+\right|^2\\\nonumber
&+&\frac{C(\Omega_2^i,d_Q,\alpha,\delta',\beta,N)}{R^{2+m(\alpha -1)+1-N}},
\end{eqnarray}
with
\begin{equation}\label{eq:Cflat}
C^\flat=C\cdot C_p(\Omega_1^i)\cdot\hat C^2(d_Q)\cdot(C_M)^\frac{2}{(2t)' }\cdot R^{(N+2q+2)\frac m t-2}.
\end{equation}
 It is here that we choose $m$ small and $\alpha$ big such that
 \begin{itemize}
 \item [$(i)$] ${\displaystyle (N+2q+2)\frac m t-2\leq -1}$;
 \\
 \item [$(ii)$] $2+m(\alpha -1)+1-N \geq 1$.
 \end{itemize}
Note that later $t$ will be fixed close to $\frac{p-1}{p-2}$.
 We point out  that  condition $(i)$ holds true for  $m $ close to zero,    since $t>1$  (see  Theorem \ref{thm: Sobolev}); instead  condition $(ii)$ is satisfied for $\alpha$ sufficiently large. \\

 It is crucial here that the choice of $m$ and $\alpha$ does not depend on $\lambda$ neither on  $d_Q$.\\

\noindent From \eqref{eq:cn6},  we have
\begin{eqnarray}\label{eq:cn6f}\\\nonumber
I_1^b&\leq&\frac{C_1 (\Omega_1^i,d_Q,\delta')}{R}\int_{\mathcal{C}(2R)}(|\nabla u|+|\nabla v|)^{p-2 }[(u-v)^+]^{\alpha -1}\left|\nabla[(u-v)^+\right|^2\\\nonumber
&+&\frac{C_2(\Omega_2^i,d_Q,\delta')}{R},
\end{eqnarray}
for some positive constants
\begin{eqnarray}\label{eq:C_1C_2}
C_1(\Omega_1^i,d_Q,\delta')\rightarrow0  \qquad &\text{if }& \quad |Q_i| \rightarrow 0 \,\,  \text{or } d_Q  \rightarrow 0 \,,\\\nonumber
C_2(\Omega_2^i,d_Q,\delta') \rightarrow0  \qquad &\text{if }& \quad |Q_i| \rightarrow 0 \,\,  \text{or } d_Q  \rightarrow 0.
\end{eqnarray}
Moreover we remark  that, for the sake of simplicity and reader convenience, we have explicited in the constants $C_1(\cdot, \cdot,\cdot)$ and  $C_2(\cdot, \cdot,\cdot)$ only the dependence on the parameter that in the sequel we are going to use.

\

Thus, by using \eqref{eq:supI_1^a} and \eqref{eq:cn6f}, equation \eqref{eq:supI_1} states as\begin{eqnarray}\label{eq:1I1}
I_1&\leq&\delta' \check C\int_{\mathcal{C}(2R)}(|\nabla u|+|\nabla v|)^{p-2}|\nabla(u-v)^+|^2
\varphi^2[(u-v)^+]^{\alpha-1}\\\nonumber
&+&\frac{C_1 (\Omega_1^i,d_Q,\delta')}{R}\int_{\mathcal{C}(2R)}(|\nabla u|+|\nabla v|)^{p-2 }[(u-v)^+]^{\alpha -1}\left|\nabla[(u-v)^+\right|^2\\\nonumber
&+&\frac{C_2(\Omega_2^i,d_Q,\delta')}{R}.
\end{eqnarray}

\

\emph{ Step 2: Evaluation  of  $I_2$.} \\

\noindent We set
 \begin{eqnarray}\label{eq:I2}
 I_2=\int_{\mathcal{C}(2R)}\frac{f(u)-f(v)}{(u-v)^+}[(u-v)^+]^{\alpha+1}\varphi^2,
 \end{eqnarray}
 and
 \begin{equation}\label{eq:alfa}
 {\displaystyle \underline{\kappa}_i=\inf_{Q_{i_{\bar\delta}}}v} \qquad\text{and} \qquad \overline{\kappa}_i=\sup_{Q_{i_{\bar\delta}}}v,\end{equation}

 where
 \[
 Q_{i_{\bar\delta}}:=\{x\in\mathbb{R}^N\,|\, \text{dist}(x\,,\,Q_i)\leq \bar\delta\}
 \]
 and $\bar\delta$ as in the statement.
 We set
 \[
 v_0^i\,:= v(x_0^i\,,\,\frac{\beta+\lambda}{2})
 \]
 recalling that $(x_0^i\,,\,\frac{\beta+\lambda}{2})$ is the center of the cube $Q_i$.
 By  Harnack inequality we have
 \begin{equation}\label{eq:alfaH}
 \underline{\kappa}_i\leq v_0^i\leq \overline{\kappa}_i\leq C_H\underline{\kappa}_i\leq C_H v_0^i\,.
 \end{equation}
 Let us consider  the two following cases:

 \

 {\bf Case 1: } $q\geq 2$.\\

 \

By Taylor expansion of $f(\cdot)$, we obtain
$$f(u)=f(v)+f'(v)(u-v)+\frac{f''(\xi)}{2}(u-v)^2,$$
 with $v<\xi<u$.
 Then \eqref{eq:I2} turns out to be
 \begin{eqnarray}\label{eq:cn8}
I_2&=&\sum_{i=1}^{\overline{N}}\int_{\mathcal{C}(2R)\cap Q_i} f'(v)[(u-v)^+]^{\alpha+1}+\sum_{i=1}^{\overline{N}}\int_{\mathcal{C}(2R)\cap Q_i}\frac{f''(\xi)}{2}[(u-v)^+]^{\alpha+2}\\\nonumber
&=&I_{2}^a+I_{2}^b,
 \end{eqnarray}
 with
 \begin{equation}\label{eq:I_2a}
 I_{2}^a:=\sum_{i=1}^{\overline{N}}\int_{\mathcal{C}(2R)\cap Q_i} f'(v)[(u-v)^+]^{\alpha+1}
 \end{equation}
 and
 $$I_{2}^b:=\sum_{i=1}^{\overline{N}}\int_{\mathcal{C}(2R)\cap Q_i}\frac{f''(\xi)}{2}[(u-v)^+]^{\alpha+2}.$$

 \

\begin{itemize}
\item[-] We start estimating $I_{2}^a$.
\end{itemize}

\

\noindent We have
 \begin{equation}\nonumber
 I_{2}^a\leq C\sum_{i=1}^{\overline{N}}(\overline{\kappa}_i)^{q-1}\int_{\mathcal{C}(2R)\cap Q_i}\left([(u-v)^+]^{\frac{\alpha+1}{2}}\right)^2\leq
 C(v_0^i)^{q-1}\sum_{i=1}^{\overline{N}}\int_{\mathcal{C}(2R)\cap Q_i}\left([(u-v)^+]^{\frac{\alpha+1}{2}}\right)^2.
 \end{equation}
Setting
\begin{equation}\label{eq:Csharp2}
C^{\sharp}=C_p(Q_i)\cdot\hat C^2(d_Q)\cdot(C^*_1)^{\frac{1}{t}}\cdot(C_M)^\frac{2}{(2t)'},
\end{equation}
(where all the constants are those given in Corollary \ref{cor:PoincarŽ}) by using the weighted Poincar\'{e} inequality given in Corollary \ref{cor:PoincarŽ}, we have
 \begin{equation}\label{eq:cn8}
 \begin{split}
 &I_{2}^a\leq C (v_0^i)^{q-1}C^{\sharp}\sum_{i=1}^{\overline{N}}\int_{\mathcal{C}(2R)\cap Q_i}|\nabla v|^{p-2}|\nabla(u-v)^+|^2[(u-v)^+]^{\alpha-1}\leq\\
& \leq C (v_0^i)^{q-1}C^{\sharp}\sum_{i=1}^{\overline{N}}\int_{\mathcal{C}(2R)\cap Q_i}( |\nabla u|+|\nabla v|)^{p-2}|\nabla(u-v)^+|^2[(u-v)^+]^{\alpha-1}\,,
\end{split}
 \end{equation}
 since $p>2$.
Considering definition \eqref{eq:Csharp2}, we shall estimate  $(v_0^i)^{q-1}(C^*_1)^\frac{1}{t}$.
By our assumption  \eqref{condcrucial}, for $\frac{p-2}{p-1}<r<1$, we can exploit   Theorem \ref{eq:weight}.
In this case our assumption  \eqref{condcrucial}  replaces the general assumption \eqref{eq:weight} in Theorem \ref{eq:weight}.
Thus we get
$$(v_0^i)^{q-1}(C^*_1)^\frac{1}{t}\leq C \beta^{[N-2p+(p-1)r-\gamma]\frac{p-2}{(p-1)r}} (v_0^i)^{q-1+[2p-2-(p-1)r-2q]\frac{p-2}{(p-1)r}}.$$
Here, we are using the relation $\tau=(p-1)r=(p-2)t$. Recall also that $\gamma=0$ if $N=2$, while if $N\geq 3$ we can take any $\gamma <N-2$,
with $\gamma $ sufficiently close to $N-2$ ($\gamma >N-2t$), according to Theorem \ref{eq:weight}.\\

\noindent For $q-1+p-2-2q\frac{p-2}{p-1}> 0$, namely (we use here the assumption $2<p<3$) for:
$$[q-(p-1)] (p-3)<0,$$
we can consequently take $r-1$ sufficiently small such that
\[
q-1+[2p-2-(p-1)r-2q]\frac{p-2}{(p-1)r}>0\,,
\]
and consequently
we get by \eqref{eq:cn8}  
\begin{equation}
\begin{split}
(v_0^i)^{q-1}C^{\sharp}&< C\,(C_M)^\frac{2}{(2t)'}\, C_p(Q_i)  (v_0^i)^{q-1+[2p-2-(p-1)r-2q]\frac{p-2}{(p-1)r}}\beta^{[N-2p+(p-1)r-\gamma]\frac{p-2}{(p-1)r}}\\
&\leq C C_p(Q_i) \beta^{[N-2p+(p-1)r-\gamma]\frac{p-2}{(p-1)r}},
\end{split}
\end{equation}
where we also used that $v_0^i\leq \|v\|_\infty\leq C$.\\
\noindent Recall now that $C_p(Q_i)$ is given by  Corollary \ref{cor:PoincarŽ} (see \eqref{estimacav}).
In particular, for  any $2< q<\tilde 2^*$ (see Remark \ref{g675}) we have
\begin{equation}\label{poinstettema}
C_p(Q_i)\leq C |Q_i|^\frac{q-2}{q}\leq C(\beta-\lambda)^{\frac{2}{p-1}\theta},
\end{equation}
where $\theta$ is any number such that $0<\theta<1$ (actually we take $\theta$ close to $1$).  \\
\noindent For $\frac{2}{p-1}>p-2$, namely (and we use again here the assumption $2<p<3$) for:
\[
p(p-3)<0
\]
we can fix $\theta$ close to $1$, $\gamma$ close to $N-2$, $r$ close to $1$, such that
 \[
 \frac{2}{p-1}\theta+[N-2p+(p-1)r-\gamma]\frac{p-2}{(p-1)r}>0
 \]
so that  we can rewrite
\eqref{eq:cn8} as follows:

\begin{eqnarray}\label{eq:cn8bis}\\\nonumber
 I_{2}^a&\leq& C_{3,a}(d_Q)\int_{\mathcal{C}(2R)}(|\nabla u|+|\nabla v|)^{p-2}|\nabla(u-v)^+|^2[(u-v)^+]^{\alpha-1}\\\nonumber
 \end{eqnarray}
with\begin{equation}\label{eq:C_3_1}
C_{3,a}(d_Q)\rightarrow0  \qquad \text{if } \quad d_Q \rightarrow 0,
\end{equation}
where in the constant, for simplicity, we have stated the dependence on $d_Q$ since it will be used in the sequel. Actually we have
$C_{3,a}(d_Q)\leq C(\beta-\lambda)^s$ for some $s>0$.

\

\begin{itemize}
\item[-] Consider now  the term $I_{2}^b$.
\end{itemize}

\

One has
\begin{eqnarray}\label{eq:cn9}
I_{2}^b&\leq&C\sum_{i=1}^{\overline{N}}\int_{\mathcal{C}(2R)\cap Q_i}\left([(u-v)^+]^\frac{\alpha+2}{p}\right)^p\\\nonumber
&\leq&C\sum_{i=1}^{\overline{N}}C_p(d_Q)\int_{\mathcal{C}(2R)\cap Q_i}[(u-v)^+]^{\alpha+2-p}|\nabla(u-v)^+|^p,
\end{eqnarray}
where we used Poincar\'e inequality in $W^{1,p}(Q_i)$, since $(u-v)^+$ is zero on $\partial \Sigma_{(\lambda,\beta)}$.  Being $u,v,\nabla u, \nabla v \in L^{\infty}(\Sigma_{(\lambda, \beta)})$, since we  have
\begin{equation}\label{eq:sup2}
|\nabla(u-v)^+|^p\leq (|\nabla u|+|\nabla v|)^{p-2}|\nabla(u-v)^+|^2,\end{equation}
from \eqref{eq:cn9} it follows
\begin{eqnarray}\label{eq:cn10}\\\nonumber
I_{2}^b&\leq&C\sum_{i=1}^{\overline{N}}C_p(d_Q)\left(||u||_{\infty}+||v||_{\infty}\right)^{3-p}\int_{\mathcal{C}(2R)\cap Q_i}\left( |\nabla u|+|\nabla v|\right)^{p-2}|\nabla (u-v)^+|^2 [(u-v)^+]^{\alpha-1}\\\nonumber
&\leq&C_{3,b}(d_Q)\int_{\mathcal{C}(2R)} \left( |\nabla u|+|\nabla v|\right)^{p-2}|\nabla (u-v)^+|^2 [(u-v)^+]^{\alpha-1},
\end{eqnarray}
where as above
\begin{equation}\label{eq:C_3_2}
C_{3,b}(d_Q)\rightarrow0  \qquad \text{if } \quad d_Q \rightarrow 0\,,
\end{equation}
and again we are using the assumption $2<p<3$.

\noindent Consider now\\

 \noindent {\bf Case 2: $p-1<q< 2$.}

$$I_2=\int_{\mathcal{C}(2R)}(f(u)-f(v))[(u-v)^+]^{\alpha}\varphi^2.$$
Consider first the  function $(\sqrt u)^{2q}$. By Taylor expansion we have
$$ u^q=v^{q}+2q(\sqrt v)^{2q-1} (\sqrt u-\sqrt v)+2q(2q-1)\xi^{2q-2}(\sqrt u-\sqrt v)^2,$$
with $\sqrt v<\xi<\sqrt u$.
Then, since $u\geq v$,  we obtain
\begin{eqnarray}\label{eq:cn12}\\\nonumber
\Big|u^q-v^q \Big | &\leq&C \Big | v^{q-\frac 12}(\sqrt u- \sqrt v)+ \xi^{2q-2}(\sqrt u - \sqrt v )^2 \Big |\\ \nonumber
&=&C\Big |v^{q-\frac 12}\frac{(u-v)^+}{(\sqrt u + \sqrt v)}+ \xi^{2q-2}\frac{[(u-v)^+]^2}{(\sqrt u + \sqrt v)^2}\Big|\\\nonumber
&\leq&C v^{q-1}(u-v)^++C[(u-v)^+]^{p-1}\frac{u^{q-1}[(u-v)^+]^{3-p}}{(\sqrt u + \sqrt v)^{2-2(3-p)}(\sqrt u + \sqrt v)^{2(3-p)}}\\\nonumber
&\leq&C v^{q-1}(u-v)^++Cu^{q+1-p}[(u-v)^+]^{p-1}.
\end{eqnarray}
for some positive constant   $C=C(q)$. In the last line of \eqref{eq:cn12} we used that, by a straightforward calculation,  one has
$$ \frac{u^{q-1}}{(\sqrt u + \sqrt v)^{2-2(3-p)}}\frac{[(u-v)^+]^{3-p}}{(\sqrt u + \sqrt v)^{2(3-p)}}\leq C u^{q+1-p}$$
with $C=C(||u||_{\infty})$ and $q>p-1$, recalling that $2<p<3$.

\

\noindent By \eqref{eq:cn12} and recalling that $\frac{f(s)-f(t)}{s^q-t^q}\leq \tilde A$ by condition ($f_1$) for $s> t$, the term $I_2$ can be estimated as follows:
\begin{eqnarray}\label{eq:cn69}\\\nonumber
I_2&\leq& C \left(\int_{\mathcal{C}(2R)}v^{q-1}[(u-v)^+]^{\alpha+1}dx+\int_{\mathcal{C}(2R)}[(u-v)^+]^{\alpha+p-1}dx\right)\\\nonumber
&\leq&C\left(\sum_{i=1}^{\overline{N}}\int_{\mathcal{C}(2R)\cap Q_i}v^{q-1}[(u-v)^+]^{\alpha+1}dx+
\sum_{i=1}^{\overline{N}}\int_{\mathcal{C}(2R)\cap Q_i}[(u-v)^+]^{\alpha+p-1}dx\right),
 \end{eqnarray}
for some positive constant $C=C(q, ||u||_{L_\infty})$.

\

\noindent Following  exactly the same calculations used for the term $I_{2}^a$ in \eqref{eq:I_2a}, we estimate the first integral on the right of  \eqref{eq:cn69} as follows
\begin{eqnarray}\label{eq:sup1}
&&
\sum_{i=1}^{\overline{N}}\int_{\mathcal{C}(2R)\cap Q_i}v^{q-1}[(u-v)^+]^{\alpha+1}dx\\\nonumber
&\leq& C(d_Q)\int_{\mathcal{C}(2R)}(|\nabla u|+|\nabla v|)^{p-2}|\nabla(u-v)^+|^2[(u-v)^+]^{\alpha-1}
\end{eqnarray}
with
$$
C(d_Q)\rightarrow0  \qquad \text{if } \quad d_Q \rightarrow 0.
$$The second integral on the right of  \eqref{eq:cn69} states as
\begin{eqnarray}\label{eq:cn13tribiss}
&&\sum_{i=1}^{\overline{N}}\int_{\mathcal{C}(2R)\cap Q_i}\left( [(u-v)^+]^\frac{\alpha +p -1}{p}\right )^pdx
\\\nonumber&\leq& \sum_{i=1}^{\overline{N}}C_p(d_Q)\int_{\mathcal{C}(2R)\cap Q_i}[(u-v)^+]^{\alpha-1}|\nabla(u-v)^+|^p\\\nonumber&\leq&
C_p(d_Q)\int_{\mathcal{C}(2R)}(|\nabla u|+|\nabla v|)^{p-2}|\nabla(u-v)^+|^2[(u-v)^+]^{\alpha-1}
\end{eqnarray}
where we used equation~\eqref{eq:sup2} and  Poincar\'e inequality in $W^{1,p}(Q_i)$ with $C_p(d_Q)\rightarrow 0$ if $d_Q\rightarrow0 $.

\

\noindent Then, in the case $p-1<q<2$ by \eqref{eq:cn69},  \eqref{eq:sup1} and \eqref{eq:cn13tribiss} for $I_2$ we have
\begin{equation}\label{eq:cn70}
I_2\leq C_{3,c}(d_Q)\int_{\mathcal{C}(2R)}(|\nabla u|+|\nabla v|)^{p-2}|\nabla(u-v)^+|^2[(u-v)^+]^{\alpha-1}dx,
\end{equation}
for some
\begin{equation}\label{eq:C_3_3}
C_{3,c}(d_Q)\rightarrow0  \qquad \text{if } \quad d_Q \rightarrow 0.
\end{equation}
Then, for any $q>(p-1)$, by equations \eqref{eq:cn8bis}, \eqref{eq:cn10}, \eqref{eq:cn70}, from \eqref{eq:cn8}
we get
\begin{equation}\label{2I2}
I_2 \leq C_3(d_Q)\int_{\mathcal{C}(2R)}(|\nabla u|+|\nabla v|)^{p-2}|\nabla(u-v)^+|^2[(u-v)^+]^{\alpha-1},
 \end{equation}
where
$$C_3(d_Q)=2\max\{C_{3,a}(d_Q),\,\,C_{3,b}(d_Q),\,\,C_{3,c}(d_Q)\}$$
and moreover, from equations \eqref{eq:C_3_1}, \eqref{eq:C_3_2} and \eqref{eq:C_3_3} one has
\begin{equation}\label{eq:C_3}
C_{3}(d_Q)\rightarrow0  \qquad \text{if } \quad d_Q \rightarrow 0.
\end{equation}

\

\emph{ Step 3: Passing to the limit and concluding  the proof.} \\
\noindent From equations \eqref{eq:cn33}, \eqref{eq:1I1} and \eqref{2I2} we obtain
\begin{eqnarray}\label{eq:L1}
&&\bar \alpha \dot  C\int_{\mathcal{C}(2R)}(|\nabla u|+|\nabla v|)^{p-2}|\nabla(u-v)^+|^2[(u-v)^+]^{\alpha-1}\varphi^2\\\nonumber
&\leq&\delta' \check C\int_{\mathcal{C}(2R)}(|\nabla u|+|\nabla v|)^{p-2}|\nabla(u-v)^+|^2
\varphi^2[(u-v)^+]^{\alpha-1}\\\nonumber
&+&\frac{C_1 (\Omega_1^i,d_Q,\delta')}{R}\int_{\mathcal{C}(2R)}(|\nabla u|+|\nabla v|)^{p-2 }|\nabla[(u-v)^+|^2[(u-v)^+]^{\alpha -1}+\frac{C_2(\Omega_2^i,d_Q,\delta')}{R} \\\nonumber&+& C_3(d_Q)\int_{\mathcal{C}(2R)}(|\nabla u|+|\nabla v|)^{p-2}|\nabla(u-v)^+|^2[(u-v)^+]^{\alpha-1}.
\end{eqnarray}
Let us choose $\delta'$ small  in \eqref{eq:L1}, say $\bar{\delta'}$ such that
\begin{itemize}
\item[$(i)$]$\tilde C=\bar \alpha \dot C - \bar{\delta'} \check C>0$\,.
\end{itemize}
Also let $R$ sufficiently large and  $d_Q$ sufficiently small  such that
\begin{itemize}
\item[$(ii)$] $\displaystyle \theta=\frac{1}{\tilde C}\left(\frac{C_1 (\Omega_1^i,d_Q,\delta')}{R}+C_3(d_Q)\right)< 2^{-N}$.
\end{itemize}
Let us set
\begin{equation}\nonumber
 \mathcal {L}(R)\,:=\,\int_{\mathcal{C}(R)}(|\nabla u|+|\nabla v|)^{p-2}|\nabla(u-v)^+|^2[(u-v)^+]^{\alpha-1}
\end{equation}
and
$$ g(R)=\frac{C_2(\Omega_2^i,\bar{d_Q},\bar{\delta'})}{\tilde C R}.$$
Then, since $u,\nabla u, v,\nabla v \in L^{\infty}(\Sigma_{(\lambda, \beta)})$, by \eqref{eq:L1}, one has
$$
\begin{cases}
\mathcal{L}(R)\leq \theta \mathcal{L}(2R)+g(R) & \forall R>0,\\
\mathcal{L}(R)\leq CR^{N} & \forall R >0,
\end{cases}
$$
and  from Lemma \ref{Le:L(R)} with $\nu=N$, since $\theta $ can by taken such that $\theta<2^{-N}$,  we get $$\mathcal{L}(R)\equiv0$$ and consequently the thesis, in the case when ($H_1$) is assumed.\\

Let us now consider the more simple case  when $\lambda>\underline{\lambda}>0$ and
 $v\geq \underline{v}>0$ in $\Sigma_{(\lambda-2\bar\delta, \beta+2\bar\delta)}$.
 In this case the constant in \eqref{condcrucial} is uniformly bounded  and \eqref{condcrucial} is:
 \begin{equation}\label{condcrucialttt}
\begin{split}
&\int_{\mathcal{K}(x'_0)} \frac{1}{|\nabla v|^{\tau}}\frac{1}{|x-y|^\gamma}\leq C\,.
\end{split}
\end{equation}
 Consequently the weighted Poinver\'{e} constant  provided by Corollary \ref{cor:PoincarŽ}, are also uniformly bounded. Therefore,
 the proof used in the previous case (when ($H_1$) is assumed) can be repeated verbatim. The fact that the weighted Poinver\'{e} constants  provided by Corollary \ref{cor:PoincarŽ} are uniformly bounded, allows to get
  \eqref{gfgfdtrscvbnmzzxs} for any $p>2$. Also the reader will guess that we only need in this case to
   estimate the term $\frac{f(u)-f(v)}{u-v}$  by a constant, and the assumption that $f$ is locally Lipschitz continuous is enough.

\begin{corollary}\label{th:wcpstripbg}
Let   $u \in C^{1, \alpha}_{loc}$ be a solution to \eqref{E:P} and assume that ($H_1$)  hold.
Let as above $\Sigma_{(\lambda, \beta)}:= \left\{ \mathbb{R}^{N-1}\times [\lambda, \beta]\right \}$ with $0\leq\lambda<\beta$.
Assume that $|\nabla u|$ is bounded and define $u_\beta$ to be the reflection of $u$ (w.r.t. the hyperplane $\{y=\beta\}$) defined by:
\[
u_\beta(x',y)\,:=\, u(x',2\beta-y)\,.
\]
Then there exists $d_0=d_0(p,u, f,N)>0$ such that, if  $0~<~(\beta-\lambda)<~ d_0$ and $ u \leq u_\beta$ on $\partial\Sigma_{(\lambda, \beta)}$, then it follows that
$$ u \leq u_\beta \qquad \text{ in } \Sigma_{(\lambda, \beta)}.$$

\noindent The same conclusion holds assuming $\Sigma_{(\lambda', \beta')}\subseteq \Sigma_{(0, \beta)}$, $ u \leq u_\beta$ on $\partial\Sigma_{(\lambda', \beta')}$ and $(\beta'-\lambda')$ sufficiently small (say $(\beta'-\lambda ')\leq d_0$).

\end{corollary}
\begin{proof}
We apply Theorem \ref{th:wcpstrip} to $u$ and $v\equiv u_\beta$, so that the condition expressed by \eqref{condcrucial}, turns to be satisfied thanks to Proposition \ref{ghhghgjfjfjbis} (see also Remark \ref{ghhghgjfjfjbisrem}).
\end{proof}

\section{Proof of Theorem \ref{mainthm}, Theorem \ref{mainthmfdfdfdfdf} and Theorem \ref{liouvillenextgenerationtris}}\label{sec666}

\noindent \emph{Proof of Theorem \ref{mainthm}}.\\
\noindent The proof of Theorem \ref{mainthm}  follows directly by  Theorem \ref{th:wcpstrip} (exactly the version given by Corollary \ref{th:wcpstripbg}), and repeating verbatim the proof of Theorem 1.3 in \cite{FMS},
by replacing the application of Theorem 1.1 in \cite{FMS} with the application of Theorem \ref{th:wcpstrip} (see Corollary \ref{th:wcpstripbg}) proved here.

We only remark that, doing this,  Theorem \ref{th:wcpstrip} has to be exploited in strips
$$\Sigma_{(0, \beta)}:= \left\{ \mathbb{R}^{N-1}\times [0, \beta]\right \},$$
where the solution $u$ is bounded since we assumed that $|\n u| \in L^{\infty}(\mathbb{R}^N_+)$ and taking into account the Dirichlet assumption.
This allows to exploit Theorem \ref{th:wcpstrip}.

Theorem \ref{th:wcpstrip} therefore applies and leads to the proof of the first part of Theorem \ref{mainthm}, that is
\[
\frac{\partial u}{\partial x_N}\,>\,0 \quad \text{in}\quad \mathbb{R}^N_+.
 \]
It follows now that $u$ has no critical points and therefore $u \in C^{2}_{loc}({\overline {{\mathbb{R}^N_+}}})$ by standard regularity theory, since
the $p$-Laplace operator is non-degenerate outside the critical points of the solution.\\

Let us now assume that $u$ is bounded and that  $N=3$ (the case $N=2$ is analogous and  has been already considered in \cite{DS3})  and let us show that  $u$ has one-dimensional symmetry with $u(x',x_N)=u(x_N)$.
We exploit  some arguments used in \cite{FSV} to which we refer for more details.
For any~$(x_1,x_2,y)\in\R^3$ and~$t\in\mathbb{R}$,
we define
\begin{equation}\label{gfgfvsvbodddd}
 u^\star(x_1,x_2,y):=\left\{
\begin{matrix}
u(x_1,x_2,y) & {\mbox{ if $y\geq 0$,}}\\
-u(x_1,x_2,-y) & {\mbox{ if $y\leq 0$}},
\end{matrix}
\right.
\end{equation}
and
$$ f^\star(t):=\left\{
\begin{matrix}
f(t) & {\mbox{ if $t\geq 0$,}}\\
-f(-t) & {\mbox{ if $t\leq 0$.}}
\end{matrix}
\right.$$
It follows, taking into account that $f(0)=0$, that
\begin{equation}\label{BequaU}
-\Delta_p u^\star=f^\star(u^\star)\qquad \text{in}\,\,\,\, \mathbb{R}^N.
\end{equation}
Moreover   $u^\star$ is monotone with $u^\star_y>0$ by construction. The conclusion follows therefore by
the $1$-D results in \cite{FSV,FSV2}. In particular by Theorem 1.1 and Theorem 1.2    in \cite{FSV} it follows that $u^\star$ (and therefore $u$) is one dimensional.\\



\noindent \emph{Proof of Theorem \ref{mainthmfdfdfdfdf}}.\\
\noindent Taking into account \eqref{gfgfgsdhjsjbvbbvbhcnncndnndhuhu}, we can apply Lemma \ref{vdcvs0987654}
to get that
\begin{equation}
\frac{\partial u}{\partial y}\geq \underline{u'}_\theta>0\qquad \text{in}\quad \Sigma_{(0,\theta)}\,,
\end{equation}
for some $\underline{u'}_\theta\,,\,\theta >0$ and $ \Sigma_{(0,\theta)}=\left\{ (x',y): x'\in \mathbb{R}^{N-1}, y\in[0,\theta]\right\}$.\\
Consequently $u(x',y)<u_{\frac{\theta}{2}}(x',y)=u(x',\theta-y) $  in $\Sigma_{(0,\frac{\theta}{2})}$ and the moving plane procedure can be started.
To conclude it is needed to repeat the proof of Theorem 1.3 in \cite{FMS}. In this case, since we already started the moving plane procedure, we only have to exploit
the weak comparison principle in narrow domains (Theorem \ref{th:wcpstrip}) far from the boundary, with $v=u_\beta$ the reflection of $u$ w.r.t.
the hyperplane $\{y=\beta\}$.
It is important now to remark that by Lemma \ref{vdcvs0987654} (see also Remark \ref{vdcvs0987654b}) $v=u_\beta$ is uniformly bounded away from zero far from the boundary and  we can exploit the second part of the statement of Theorem \ref{th:wcpstrip} which allows the result to hold for  $f$  positive($f(s)>0$ for $s>0$) and locally Lipschitz continuous and for any $p>2$.\\
If ($H_2$) holds, we can exploit Lemma \ref{le:cuccurucucu} to deduce  \eqref{gfgfgsdhjsjbvbbvbhcnncndnndhuhu} and the thesis by the above arguments.\\

\noindent \emph{Proof of Theorem \ref{liouvillenextgenerationtris}}.\\
\noindent The proof of the first part of Theorem \ref{liouvillenextgenerationtris} (that is $u=0$ if $q<q_c(N,p)$)  follows directly by Proposition 2.3 in \cite{DFSV}, recalling that monotone solutions are also stable solutions. Equivalently we can also apply Theorem 1.5 in \cite{DFSV}
if we assume that $u$ is defined in the whole space by odd reflection as in \eqref{gfgfvsvbodddd}. \\

\noindent To prove the second part of Theorem \ref{liouvillenextgenerationtris} (that is $u=0$ if $q<q_c((N-1),p)$) we argue as in Theorem 12 of \cite{Fa2} and we assume by contradiction that $u$ is not identically zero. Therefore, $u>0$ in $\mathbb{R}^N_+$ by the strong maximum principle \cite{V}. Also for simplicity we assume that $u$ is defined in the whole space by odd reflection as in \eqref{gfgfvsvbodddd}. \\
Consequently we can exploit  Theorem \ref{mainthm}  and get that $u$ is monotone increasing with $\frac{\partial u}{\partial x_N}\,>\,0 $ in $\mathbb{R}^N_+$.
Since $u$ is bounded by assumption in this case, we can  define
\begin{equation}\begin{split}\label{w77w8ss88qq}
& w(x')
:=\lim_{t\rightarrow \infty} u(x',y+t).
\end{split}\end{equation}
The limit in
\eqref{w77w8ss88qq}
holds in
$C^{1}_{\rm loc}(\R^{N-1})$
and $w$
is a bounded
weak solution of
\begin{equation}\label{a766321-2}
-\Delta_p w=w^q\qquad \text{in} \qquad \R^{N-1},
\end{equation}
see for example \cite{FSV}.
Here below we will show that $w$ is \emph{stable}  so that
the thesis $u=0$ will follow  by Theorem 1.5 in \cite{DFSV} applied in $\mathbb{R}^{N-1}$.\\

Let us therefore show that $w$ is stable, that is
\begin{equation}\label{linearizedgeneral}
\begin{split}
  L_w(\phi,\phi)&=\int_\Omega |\nabla w|^{p-2}|\nabla\phi|^2+(p-2)\int_\Omega |\nabla w|^{p-4}(\nabla
w,\nabla\phi)^2 - \int_\Omega q\,w^{q-1}\phi^2 \geq 0
\end{split}
\end{equation}
for any $\phi\in C^\infty_c(\mathbb{R}^{N-1})$. We set
\[
u^t(x',y)\,:=\, u(x',y+t).
\]
Since $u$ is monotone so does $u^t$ for any $t\in\mathbb{R}$, consequently  $u^t$  is also stable for any $t\in\mathbb{R}$ (see \cite{DFSV}), and therefore $L_{u^t}(\varphi,\varphi)\geq 0$
for any $\varphi\in C^\infty_c(\mathbb{R}^{N})$. \\

We  then take~$\varphi:= \varphi_1 (x')
\varphi_2(y)$, with~$\varphi_1\in C^\infty_c(B'_R)$ where
$B'_R$ is the ball of radius $R$ in $\mathbb{R}^{N-1}$ centered at zero,
and~$\varphi_2:\R\rightarrow\R$
of the form~$\varphi_2(y):=\sqrt\mu \tau(\mu y)$, where~$\mu>0$ is a small parameter,~$\tau
\in C^\infty_0(\R)$ and
$ \int_\R \tau^2 (y) dx_N= 1\,$, so
 that
\begin{equation}\label{ASahay7q7q8Saq1p2}
\int_\R \varphi_2^2 (y) dx_N= 1\,.
\end{equation}
The stability condition for $u^t$ reads as
\begin{equation}\label{GYIHBHGCjghvchvxj}
 L_{u^t}(\varphi_1 (x')
\varphi_2(y),\varphi_1 (x')
\varphi_2(y))\geq 0\,.
\end{equation}
Note now that $q>p-1$ and $p>2$ implies that $qs^{q-1}$ is $C^1$.  \\
\noindent This and the assumption $p>2$, together with the fact that
the limit in
\eqref{w77w8ss88qq}
holds in
$C^{1}_{\rm loc}(\R^{N-1})$, gives that $\big||\nabla u|^{p-2}-|\nabla w|^{p-2}\big|$ and $\big|u^{q-1}-w^{q-1}\big|$
are uniformly small in the cylinder  $B'_R\times supp \, (\varphi_2)$ for $t=t(\mu)$ large. Consequently by the stability condition \eqref{GYIHBHGCjghvchvxj} and some elementary calculations we get
\[
0\leq L_u(\varphi_1 (x')
\varphi_2(y)\,,\,\varphi_1 (x')
\varphi_2(y))= L_w(\varphi_1\,,\,\varphi_1)\,+\, r(\mu\,,\,t)\,,
\]
where $r(\mu\,,\,t)$ can be taken arbitrary small for $\mu $ small and  $t=t(\mu)$ large. This shows that  $L_w(\varphi_1\,,\,\varphi_1)\geq 0$ and and the stability of $w$ in $\mathbb{R}^{N-1}$.\\

\noindent Let us now prove the last part of the thesis, and assume that  $p>2$, with $f(\cdot)$ satisfying ($f_2$). We deduce again by Theorem \ref{mainthm}  that $u$ is monotone increasing with $\frac{\partial u}{\partial x_N}\,>\,0$ in $\mathbb{R}^N_+$. We therefore define $w$ as above and get the thesis by the fact that $w=0$ by \cite{MP}.

\

\noindent{\bf Acknowledgements.}
The authors would like to thank the anonymous referees for their useful comments and suggestions.

\bigskip

\end{document}